\documentclass[11pt]{amsart}

\usepackage{epigamath}

\usepackage[english]{babel}

\numberwithin{equation}{section}

\usepackage[shortlabels]{enumitem}
\setlist[enumerate,1]{label={\rm(\arabic*)}, ref={\rm\arabic*}}
\setlist[enumerate,2]{label={\rm(\alph*)}, ref={\rm\alph*}}

\usepackage{url}

\usepackage{csquotes}
\usepackage[all]{xy}
\usepackage{verbatim}
\usepackage{calligra}
\usepackage{subfigure}
\usepackage{needspace}

\usepackage{amssymb, amsmath, mathrsfs}
\usepackage{mathtools}

\usepackage{tikz}
\usepackage{tikz-cd}

\theoremstyle{plain}
\newtheorem{theorem}{Theorem}[section]

\newtheorem{proposition}[theorem]{Proposition}
\newtheorem{lemma}[theorem]{Lemma}

\theoremstyle{definition}
\newtheorem{definition}[theorem]{Definition}

\theoremstyle{remark}

\newtheorem{remark}[theorem]{Remark}

\DeclareMathOperator{\Log}{\mathcal{L}{\kern -2pt {\it og}}\,}

\newcommand{\CC}{\mathbb{C}}

\newcommand{\PP}{\mathbb{P}}
\renewcommand{\AA}{\mathbb{A}}
\newcommand{\QQ}{\mathbb{Q}}
\newcommand{\ZZ}{\mathbb{Z}}

\newcommand{\GG}{\mathbb{G}}

\newcommand{\Li}{\mathrm{Li}}
\newcommand{\dR}{\mathrm{dR}}
\newcommand{\B}{\mathrm{B}}

\DeclareMathOperator{\id}{id}
\DeclareMathOperator{\Hom}{Hom}
\DeclareMathOperator{\coker}{coker}
\DeclareMathOperator{\Gr}{Gr}
\DeclareMathOperator{\Ind}{Ind}
\DeclareMathOperator{\Res}{Res}
\DeclareMathOperator{\Ext}{Ext}
\DeclareMathOperator{\Lie}{Lie}

\DeclareMathOperator{\Spec}{Spec}
\DeclareMathOperator{\Sym}{Sym}

\newcommand{\abs}[1]{\lvert#1\rvert}

\newcommand{\DM}{\mathbf{DM}}
\newcommand{\DMT}{\mathbf{DMT}}
\newcommand{\MT}{\mathbf{MT}}
\newcommand{\sgn}{\mathrm{sgn}}
\newcommand{\mot}{\mathrm{mot}}
\newcommand{\an}{\mathrm{an}}

\let\leq\leqslant
\let\geq\geqslant

\renewcommand{\i}{\mathrm{i}}
\renewcommand{\H}{\mathrm{H}}
\newcommand{\M}{\mathrm{M}}
\renewcommand{\d}{\mathrm{d}}
\renewcommand{\t}{\mathrm{t}}
\renewcommand{\c}{\mathrm{c}}

\newcommand{\diagram}[1]{\SelectTips{lu}{10}\xymatrix{#1}}

\newcommand{\polH}{\mathcal{L}_n^{\H}}
\newcommand{\indpolH}{\mathcal{L}^{\H}}
\newcommand{\KummerH}{\mathcal{K}^{\H}}

\newcommand{\To}{\longrightarrow}

\renewcommand{\varnothing}{\emptyset}

\def\lowsim{\vbox to 0pt{\vss\hbox{$\scriptstyle\sim$}\vskip-1.5pt}}
\def\lowsimeq{\vbox to 0pt{\vss\hbox{$\scriptstyle\simeq$}\vskip-1.5pt}}

\DeclareRobustCommand\longhookrightarrow
    {\lhook\joinrel\relbar\joinrel\rightarrow}
    
\DeclareRobustCommand\longtwoheadrightarrow
    {\relbar\joinrel\twoheadrightarrow}

\definecolor{Turquoise4}{RGB}{0,134,139}
\definecolor{VioletRed3}{RGB}{205,50,120}

\newcommand{\supth}[1]{\ensuremath{#1^{\mathrm{th}}}}
\newcommand{\lra}{\longrightarrow}
   
\EpigaVolumeYear{9}{2025} \EpigaArticleNr{3} \ReceivedOn{July 8, 2023}
\InFinalFormOn{April 1, 2024}
\AcceptedOn{June 9, 2024}

\title{A construction of the polylogarithm motive}
\titlemark{A construction of the polylogarithm motive}

\author{Cl\'ement Dupont}
\address{Institut Montpelli\'erain Alexander Grothendieck, Universit\'{e} de Montpellier, CNRS, Montpellier, France}
\email{clement.dupont@umontpellier.fr}
\author{Javier Fres\'an}
\address{Sorbonne Universit\'e and Universit\'e Paris Cit\'e, CNRS, IMJ-PRG, 75005 Paris, France}
\email{javier.fresan@imj-prg.fr}

\authormark{C.~Dupont and J.~Fres\'an}

\AbstractInEnglish{Classical polylogarithms give rise to a variation of mixed Hodge--Tate structures on the punctured projective line $S=\PP^1_\ZZ\setminus \{0, 1, \infty\}$, which is an extension of the symmetric power of the Kummer variation by a trivial variation. By results of Beilinson--Deligne, Huber\nobreakdash--Wildeshaus, and Ayoub, this polylogarithm variation has a lift to the category of mixed Tate motives over $S$, whose existence is proved by computing the corresponding space of extensions in both the motivic and the Hodge settings. In this paper, we construct the polylogarithm motive as an explicit relative cohomology motive, namely that of the complement of the hypersurface \hbox{$\{1-zt_1\cdots t_n=0\}$} in affine space~$\mathbb{A}^n_S$ relative to the union of the hyperplanes $\{t_i=0\}$ and $\{t_i=1\}$.}



\acknowledgement{During the preparation of this work, the authors were partially supported by the grant ANR-18-CE40-0017 of the Agence Nationale de la Recherche.}

\begin{document}


\removebelow{6pt}

\maketitle

\begin{prelims}

\DisplayAbstractInEnglish

\bigskip

\DisplayKeyWords

\medskip

\DisplayMSCclass

\end{prelims}


\newpage

\setcounter{tocdepth}{1}

\tableofcontents


\section{Introduction}

\subsection{The polylogarithm variation of mixed Hodge--Tate structures}

Let $n \geq 1$ be an integer. The~$\supth{n}$ \textit{polylogarithm} $\Li_n$ is the function defined on the complex unit disk $\abs{z}<1$ by 
\[
\Li_n(z)=\sum_{k=1}^\infty \frac{z^k}{k^n}.
\]
That is, $\Li_1(z)=-\log(1-z)$ and $\Li_n(z)$ is, for $n\geq 2$, the primitive of $\Li_{n-1}(z)/z$ that vanishes at~\hbox{$z=0$}. Hence, the vector $\left(1, \Li_1(z), \dots, \Li_n(z)\right)$ is a solution of the system of linear differential equations $\d L=L\,\Omega_n$ on the punctured Riemann sphere $\PP^1(\CC) \setminus \{0, 1, \infty\}$, where $\Omega_n$ is the matrix
   \begin{equation}\label{}
\setlength{\arraycolsep}{4pt}\def\arraystretch{1}
			\Omega_n = \left(\begin{matrix}
			0 & \frac{\d z}{1-z} & & &     &&\\
			 & 0   & \frac{\d z}{z}& &  & 0 &  \\
			 &      & 0    & & \ddots &  & \\
			 &      &     & & \ddots &  & \\
			 &0&  & & &   0 & \frac{\d z}{z}   \\
			 && & && & 0  \\
			\end{matrix}\right).
   \end{equation} 
  A full basis of fundamental solutions is given by the rows of the matrix 
\begin{equation}\label{eqn:matrixpolylog}
\setlength{\arraycolsep}{4pt} \def\arraystretch{1.3}
 \Lambda_n(z)=\left(\begin{matrix} 
			1 & \Li_1(z) & \Li_2(z)        & \Li_3(z) & \cdots & \Li_n(z)\\
			 & 2\pi \i   & 2\pi\i\,\log(z)  & 2\pi \i \frac{\log^2(z)}{2}& \cdots & 2\pi \i\frac{\log^{n-1}(z)}{(n-1)!}\\[.5ex]
			 &            & (2\pi \i)^2      & (2\pi \i)^2 \log(z) & \cdots & (2\pi \i)^2 \frac{\log^{n-2}(z)}{(n-2)!} \\
			 &  &  &  & &  \\
			 & & & \ddots & & \vdots \\
			& & 0 & &  &  \\
			& & & & &  (2\pi \i)^n 
			\end{matrix}\right). 
\end{equation} 
The entries of $\Lambda_n(z)$ are multivalued functions on $\PP^1(\CC) \setminus \{0, 1, \infty\}$, and analytic continuation along a loop around one of the punctures left-multiplies $\Lambda_n(z)$ by a monodromy matrix. These monodromy matrices, first computed by Ramakrishnan~\cite{ramakrishnanmonodromy}, are upper triangular with $1$s along the diagonal and have rational entries thanks to the normalization by powers of $2\pi\i$. 

Deligne \cite{DeltoBloch} realized that the matrix $\Lambda_n(z)$ gives rise to a variation of mixed Hodge--Tate structures on \hbox{$\PP^1(\CC)\setminus\{0,1,\infty\}$}, the $\supth{n}$ \emph{polylogarithm variation} $\polH$, which is defined as follows:  
\begin{itemize}[leftmargin=10pt, topsep= 0pt, itemsep=2pt]
\item Its underlying holomorphic vector bundle is trivial of rank $n+1$ with basis $e_0,\ldots, e_n$, equipped with the flat connection $\d+\Omega_n$. Its weight and Hodge filtrations are such that $W_{2k}=W_{2k+1}$ is spanned by $e_0,\ldots, e_k$, and $F^k$ by $e_k,\ldots, e_n$, for all $k$.
\item Its underlying $\QQ$-local system consists of those holomorphic functions 
\[
\varphi\colon U \longrightarrow \CC e_0\oplus\cdots \oplus \CC e_n \qquad \left(U\subset \PP^1(\CC)\setminus \{0,1,\infty\}\right) 
\] such that $\Lambda_n(z)\varphi(z)$ has locally constant rational entries. The weight filtration is such that \hbox{$W_{2k}=W_{2k+1}$} consists of those $\varphi$ for which $\Lambda_n(z)\varphi(z)$ takes values in $\CC e_0\oplus\cdots\oplus \CC e_k$, which defines a sub-local system by the special shape of the monodromy matrices.
\end{itemize}
Concretely, the fiber of $\polH$ at $z\in \PP^1(\CC) \setminus \{0, 1, \infty\}$ can also be described as the $\QQ$-vector space of dimension $n+1$ with basis $e_0,\ldots,e_n$, with weight and Hodge filtrations such that \hbox{$W_{2k}=W_{2k+1}$} is spanned by $e_0,\ldots,e_k$, and $F^k$ is spanned by the $\supth{k}$ through $\supth{n}$ columns of $\Lambda_n(z)$, for all $k$. 

The block-triangular shape of \eqref{eqn:matrixpolylog} shows that $\polH$ contains the trivial variation with fiber~$\QQ(0)$ as a subobject, and that the quotient is a Tate twist (corresponding to the multiplicative factor~$2\pi\i$) of the $(n-1)$-symmetric power of the \emph{Kummer variation}
$\KummerH$, described by the same procedure as above starting from the matrix
\[
\left(\begin{matrix}
   1 & \log(z) \\
   0 & 2\pi \i 
\end{matrix}\right). 
\] 
It is also apparent from the shape of \eqref{eqn:matrixpolylog} that $\polH$ contains $\mathcal{L}_{n-1}^{\H}$ as a subobject, and hence we have an inductive system $\indpolH$ of variations of mixed Hodge--Tate structures. The symmetric powers $\Sym^n(\KummerH)$ also make up an inductive system induced by the inclusion of $\QQ(0)$ inside $\KummerH$, and we get a short exact sequence of ind-variations of mixed Hodge--Tate structures
\begin{equation}\label{eq: short exact sequence polylog hodge} 
0 \longrightarrow \QQ(0) \longrightarrow \indpolH \longrightarrow \Sym\left(\KummerH\right)(-1)\longrightarrow 0.
\end{equation}

\begin{remark}\label{rem: dual}
What usually appears in the literature, see \cite{Ramakrishnansurvey, Hainpolylog}, is the dual variation $(\polH)^\vee$, which is less natural from a cohomological viewpoint because it has non-positive weights. Its underlying holomorphic vector bundle is trivial of rank $n+1$, with basis $f_0,\ldots,f_n$, equipped with the flat connection $\d-{}^{\t}\Omega_n$, and with weight and Hodge filtrations such that $W_{-2k}=W_{-2k+1}$ is span\-ned by $f_k,\ldots,f_n$, and $F^{-k}$ by $f_0,\ldots,f_k$, for all $k$. Its underlying local system is the $\QQ$-span of the rows of $\Lambda_n(z)$, with weight filtration such that~$W_{-2k}=W_{-2k+1}$ is the $\QQ$-span of the $\supth{k}$ through $\supth{n}$ rows. Our descriptions of $\polH$ and $(\polH)^\vee$ are related by the fact that the rows of $\Lambda_n(z)$ express the coordinates of the dual basis $e_0^\vee,\ldots,e_n^\vee$ in the basis $f_0,\ldots,f_n$. 
\end{remark}

\subsection{The polylogarithm motive}

In an attempt to find a motivic interpretation of Zagier's conjecture, see \cite{zagierpolylogarithms}, expressing the special values of the Dedekind zeta function of a number field in terms of polylogarithms, Beilinson and Deligne \cite{beilinsondeligneinterpretation} postulated the existence of a lift of the polylogarithm variation of mixed Hodge--Tate structures to the then-conjectural abelian category of \emph{mixed Tate motives} with rational coefficients over $\PP^1_\QQ\setminus\{0,1,\infty\}$. The formalism of motivic polylogarithms was generalized to multiple polylogarithms by Goncharov, leading to progress on Zagier's conjecture; see~\cite{goncharovpolylogsarithmeticgeometry}. We refer the reader to the survey article \cite{dupontbourbaki} for more details on the motivic aspects of Zagier's conjecture.

We work over the base scheme $S=\PP^1_\ZZ\setminus \{0,1,\infty\}$. By the work of Voevodsky \cite{voevodskytriangulated}, we now have access to a triangulated category $\DM(S)$ of mixed motives with rational coefficients over~$S$. For this particular choice of base $S$, one can extract from $\DM(S)$ an abelian category $\MT(S)$ of \emph{mixed Tate motives} with rational coefficients over $S$ as in the case where the base is a number field, explained by Levine \cite{levinetatemotives}.
Inspired by the constructions of Wildeshaus and Huber\nobreakdash--Wildeshaus in the Hodge and the $\ell$-adic settings \cite{wildeshaus, HW}, Ayoub \cite{AyoubOberwolfach} defined a polylogarithm motive as an ind\nobreakdash-object of~$\MT(S)$. The idea is to compute the extension group
$$\Ext^1_{\Ind\left(\MT(S)\right)}\left(\Sym(\mathcal{K})(-1), \QQ_S(0) \right)$$
and define the polylogarithm motive as a specific extension class (see Appendix~\ref{sec: appendix ext} for more details on this computation and a precise comparison with that of \cite{AyoubOberwolfach}). Note that the references \cite{wildeshaus, HW, AyoubOberwolfach} place themselves in a \emph{dual} setting, consistently with Remark~\ref{rem: dual}.

In this paper, we give an explicit construction of the polylogarithm motive as the relative cohomology motive (see Definition~\ref{defi: appendix relative cohomology motive} for this notion) of a pair of varieties over $S$. Our starting point is the integral representation 
\begin{equation}\label{eq:intrep1}
\Li_n(z)=\int_{[0, 1]^n} \frac{z\, \d t_1\cdots \d t_n}{1-zt_1\cdots t_n}, 
\end{equation} valid for $z$ outside the half-line $[1, \infty)$, which suggests to work in the following geometric framework. Let~$z$ denote the coordinate on $S$, and let $X_n=\mathbb{A}^n_S$ be the affine $n$-space over $S$ with coordinates $t_1, \ldots, t_n$.  Consider the closed $S$-subschemes of $X_n$ defined by the equations
\begin{displaymath}
    A_n= \{1-zt_1\cdots t_n=0\} \quad\textnormal{and}\quad B_n = \{t_1(1-t_1)\cdots t_n(1-t_n)=0\},
\end{displaymath} so that the integrand of \eqref{eq:intrep1} defines an algebraic differential $n$-form on $X_n \setminus A_n$, the integration domain a singular $n$-chain in $X_n(\mathbb{C})$ with boundary in $B_n(\mathbb{C})$, and the polylogarithm $\Li_n(z)$ a period function of the family of relative cohomology groups
\begin{equation}\label{eqn:relativecohom}
\H^n(X_n \setminus A_n, B_n\setminus A_n \cap B_n).
\end{equation} Figure~\ref{fig1}  illustrates the case $n=2$.

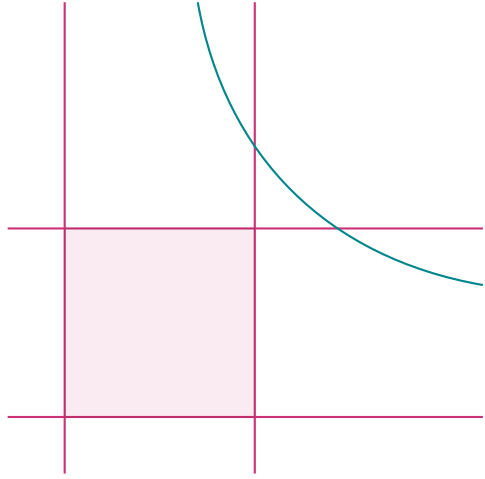
\begin{figure}
\begin{center}
	\begin{tikzpicture}
	\def\a{2.5};
	\draw[thick,VioletRed3] (0,-.3*\a) -- (0,2.2*\a);
	\draw[thick,VioletRed3] (\a,-.3*\a) -- (\a,2.2*\a);
	\draw[thick,VioletRed3] (-.3*\a,0) -- (2.2*\a,0);
	\draw[thick,VioletRed3] (-.3*\a,\a) -- (2.2*\a,\a);
	\draw[fill=VioletRed3, opacity=0.1] (0,0) -- (0,\a) -- (\a,\a) -- (\a,0) -- cycle;
	\path[draw, color=Turquoise4, thick] (.7*\a, 2.2*\a) to[out = -80, in = 170] (2.2*\a,.7*\a);
  \node[anchor=north] at (2*\a,.7*\a) {{\color{Turquoise4}{$A_2$}}};
\node[anchor=north] at (.15*\a,1.3*\a) {{\color{VioletRed3}{$B_2$}}};
	\end{tikzpicture}
	\end{center}	
 \caption{The hypersurfaces $A_2$ and $B_2$ in the affine plane $X_2$.}
  \label{fig1}
\end{figure}

\begin{definition}
The $\supth{n}$ \emph{polylogarithm motive} is the relative cohomology motive 
\[
\mathcal{L}_n=\M\left(X_n \setminus A_n , B_n \setminus A_n \cap B_n\right)[n].
\]
\end{definition}

\textit{A priori} an object of $\DM(S)$, the $\supth{n}$ polylogarithm motive is an iterated extension of the pure Tate motives $\QQ_S(-k)$ which turns out to lie in the abelian subcategory $\MT(S)$ of mixed Tate motives over~$S$, reflecting the fact that the cohomology groups of the pair $(X_n \setminus A_n, B_n\setminus A_n \cap B_n)$ are concentrated in degree~$n$. 
Besides, a partial boundary morphism (see Definition~\ref{defi: partial boundary morphism}) along the irreducible component $\{t_n=1\}$ of $B_n$ gives rise to a morphism of motives $\mathcal{L}_{n-1} \to \mathcal{L}_n$ that makes up an inductive system $\mathcal{L}$ in~$\MT(S)$, in other words, an object of the ind-category $\Ind\left(\MT(S)\right)$. 

Our main result is as follows (see Theorems~\ref{thm: short exact sequence L},~\ref{thm: short exact sequence L inductive},~\ref{thm: hodge realization polylog} below).

\begin{theorem} The ind-motive $\mathcal{L}$ fits into a short exact sequence 
    \begin{equation}\label{eq: short exact sequence polylog intro}
0\longrightarrow \QQ_S(0) \longrightarrow \mathcal{L} \longrightarrow \Sym(\mathcal{K})(-1) \longrightarrow 0
\end{equation} 
in the category $\Ind\left(\MT(S)\right)$. Its Hodge realization is the polylogarithm ind-variation $\indpolH$. 
\end{theorem}

It is easy to show that $\mathcal{L}_n$ fits into a short exact sequence
\begin{equation}\label{eq: short exact sequence a priori} 0\longrightarrow \QQ_S(0) \longrightarrow \mathcal{L}_n \longrightarrow \M(A_n,A_n\cap B_n)[n-1](-1) \longrightarrow 0,
\end{equation}
where $\M(A_n, A_n\cap B_n)$ is again a relative cohomology motive, and the crux of the proof (Theorem~\ref{thm: structureLog} below) consists in establishing an isomorphism 
\begin{equation}\label{eq: crucial isomorphism}
\M(A_n,A_n\cap B_n)[n-1] \simeq \Sym^{n-1}(\mathcal{K}).
\end{equation}
The main technical ingredient in the proof is a motivic lift of a spectral sequence originally due to Getzler \cite{getzlerMHM} in the Hodge setting, which computes motives of configuration spaces with coefficients and is a special case of a general construction of \cite{dupontjuteau}. As the referee pointed out to us, the isomorphism \eqref{eq: crucial isomorphism} was established using a different language and in a more abstract setting by Levine \cite[Proposition 9.3.3]{levinetubular} and Ayoub \cite[Theorem 3.6.44]{ayoubPhD2}.

The short exact sequence \eqref{eq: short exact sequence a priori} was already noticed by Deligne in a letter to Beilinson \cite{DeltoBeil}, where the isomorphism \eqref{eq: crucial isomorphism} is conjectured:
\begin{quote}
    \emph{$[\cdots]$ while $\H^{n-1}\left(\prod_{1}^n x_i=z, \mathrm{rel}\, x_i=1\right)$ is the $\Sym^{n-1}$ of the Kummer extension $[\cdots]$. At least I am convinced it is, but here also I would like to understand why.}
\end{quote} The letter was prompted by Ball and Rivoal's theorem \cite{rivoalcras, ballrivoal} according to which the Riemann zeta function takes irrational values at infinitely many odd integers. Their proof features integrals of the form
\begin{equation}\label{eq: ball rivoal integrals}
\int_{[0,1]^n}\frac{ t_1^{u_1}(1-t_1)^{v_1}\cdots t_n^{u_n}(1-t_n)^{v_n}}{(1-zt_1\cdots t_n)^r}\,\d t_1\cdots \d t_n
\end{equation}
for integer parameters $u_i,v_i,r$. By elementary manipulations, these integrals can be written as linear combinations with polynomial coefficients of $1$ and the polylogarithms $\Li_1(z), \ldots, \Li_n(z)$. As they are period functions of the family of relative cohomology groups \eqref{eqn:relativecohom}, the conceptual explanation is that these groups are incarnations of the polylogarithm motives. Ball and Rivoal were eventually interested in the evaluations of \eqref{eq: ball rivoal integrals} at $z=1$, for which a geometric interpretation was studied by the first-named author \cite{dupontoddzeta}. The present paper can therefore be thought of as a functional version of \emph{op.\,cit.}. 

An advantage of identifying the polylogarithm motive with an explicit relative cohomology motive is that one can then define it in the category of perverse Nori motives over $S$, see \cite{ivorra-morel}, where the computation of extension groups is currently out of reach.

\begin{remark}\label{rem: huber-kings}
Huber and Kings \cite{huberkings} produced, for every smooth group scheme $G$ over a base, a \emph{polylogarithm extension class}, which is \eqref{eq: short exact sequence polylog intro} in the case of $G=\mathbb{G}_m$ over~$\Spec(\ZZ)$ (see Appendix~\ref{sec: appendix ext} for a more precise discussion). It is unclear to us how to adapt our methods to produce geometric constructions for those extension classes beyond the case of $\mathbb{G}_m$, even in the case of the elliptic polylogarithm of Beilinson--Levin \cite{beilinsonlevin}, corresponding to elliptic curves. 
\end{remark}
 
\subsection{Iterated integrals and the motivic fundamental group}

Considering instead the more familiar representation of the polylogarithm as the iterated integral 
\begin{equation}\label{eqn:simplicialLin}
    \Li_n(z)=\int_{0 \leq x_1 \leq \cdots \leq x_n \leq 1} \frac{z\,\d x_1}{1-zx_1}\frac{\d x_2}{x_2}\cdots \frac{\d x_n}{x_n},  
\end{equation} which is related to \eqref{eq:intrep1} through the change of variables 
\begin{equation}\label{eqn:changeofvar}
    (x_1, \dots, x_n)=(t_1t_2\cdots t_n, t_2\cdots t_n, \dots, t_{n-1}t_n, t_n),
\end{equation}
one is led to work in a slightly different geometric framework. Namely, one considers the closed $S$-subschemes of $X_n'=\AA^n_S$ defined by 
\[
A_n'=\{(1-zx_1)x_2\cdots x_n=0\} \hspace{.3cm}\textnormal{ and }\hspace{.3cm} B'_n=\{x_1(x_2-x_1)(x_3-x_2)\cdots (x_n-x_{n-1})(1-x_n)=0\}. 
\]
However, the integral \eqref{eqn:simplicialLin} is not a period of the relative cohomology group 
\[
\H^n(X_n' \setminus A_n', B_n' \setminus A_n' \cap B'_n) 
\]
since the integration simplex $\Delta_n$ of \eqref{eqn:simplicialLin} meets the subvariety $A_n'$ on $\{x_1=x_2=0\}$. The trick to separate them is to resort to a tower of blow-ups: Let $\pi_n \colon \widetilde{X}'_n \to X_n'$ be the composition of 

\begin{itemize}[leftmargin=10pt, topsep= 0pt, itemsep=2pt, partopsep=3pt]
    \item the blow-up of $X'_n$ at the origin,
    \item the blow-up along the strict transform of the line $\{x_1=x_2=\cdots=x_{n-1}=0\}$,
    \item the blow-up along the strict transform of the plane $\{x_1=x_2=\cdots=x_{n-2}=0\}$,
    
    \item[$\vdots$]
    
    \item the blow-up along the strict transform of the codimension $2$ subspace $\{x_1=x_2=0\}$,
\end{itemize} and let $\widetilde{A}'_n$ and $\widetilde{B}'_n$ denote the strict transforms of $A'_n$ and~$B'_n$, respectively, and $E_n$ the exceptional divisor of~$\pi_n$. Then the boundary of the preimage by $\pi_n$ of the interior of $\Delta_n$ lies on $\widetilde{B}'_n\cup E_n$ and does not meet~$\widetilde{A}'_n$, as shown in Figure~\ref{fig2}  for $n=2$. 

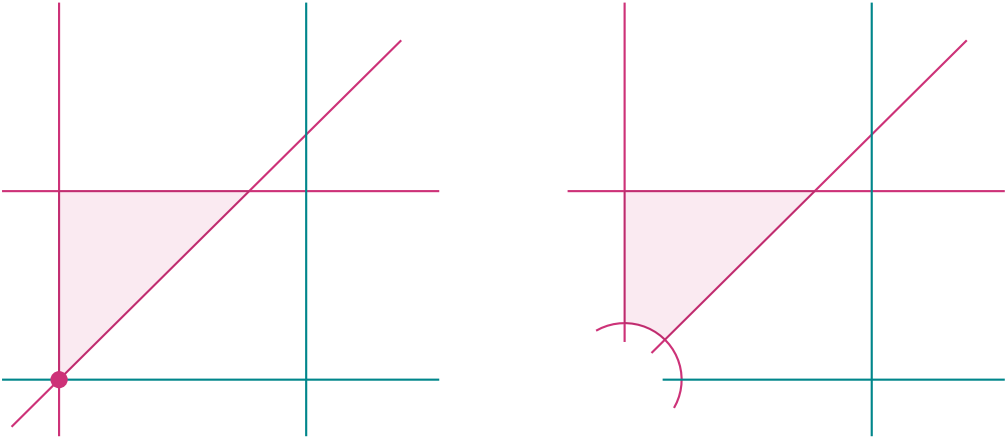
\begin{figure}[ht]
\centering 
\subfigure{\begin{tikzpicture}
	\def\a{2.5};
	\draw[thick,VioletRed3] (0,-.3*\a) -- (0,2*\a);
	\draw[thick,VioletRed3] (-.25*\a,-.25*\a) -- (1.8*\a,1.8*\a);
	\draw[thick,VioletRed3] (-.3*\a,\a) -- (2*\a,\a);
	\draw[fill=VioletRed3, opacity=0.1] (0,0) -- (\a,\a) -- (0,\a) -- cycle;
	\draw[thick,Turquoise4] (-.3*\a,0) -- (2*\a,0);
	\draw[thick,Turquoise4] (1.3*\a,-.3*\a) -- (1.3*\a,2*\a);
	\node[circle,inner sep=2pt,draw=VioletRed3,fill=VioletRed3, thick] (root) at (0,0) {};
   \node[anchor=north] at (1.5*\a, .3*\a) {{\color{Turquoise4}{$A_2'$}}};
   \node[anchor=north] at (.2*\a, 1.4*\a) {{\color{VioletRed3}{$B_2'$}}};
	\end{tikzpicture}}
 \hspace{1.5cm} 
 \subfigure{\begin{tikzpicture}
	\def\a{2.5};
	\draw[thick,VioletRed3] (0,.2*\a) -- (0,2*\a);
	\draw[thick,VioletRed3] (45:.2*\a) -- (1.8*\a,1.8*\a);
	\draw[thick,VioletRed3] (-.3*\a,\a) -- (2*\a,\a);
	\draw [VioletRed3,thick,domain=-30:120] plot ({.3*\a*cos(\x)}, {.3*\a*sin(\x)});
	\draw[fill=VioletRed3, opacity=0.1] (45:.3*\a) arc (45:90:.3*\a) -- (0,.3*\a) -- (0,\a) -- (\a,\a) -- cycle;
    (90:3mm) arc (90:0:3mm) arc (180:90:3mm) ;
	\draw[thick,Turquoise4] (.2*\a,0) -- (2*\a,0);
	\draw[thick,Turquoise4] (1.3*\a,-.3*\a) -- (1.3*\a,2*\a);
 \node[anchor=north] at (1.5*\a, .3*\a) {{\color{Turquoise4}{$\widetilde{A_2}'$}}};
   \node[anchor=north] at (.2*\a, 1.4*\a) {{\color{VioletRed3}{$\widetilde{B_2}'$}}};
   \node[anchor=north] at (.05*\a, .1*\a) {{\color{VioletRed3}{$E_2$}}};
	\end{tikzpicture}}
 \caption{In $\mathbb{A}^2$, the blow-up of the origin separates the boundary of the integration simplex $\{0\leq x_1\leq x_2\leq 1\}$ from the pole divisor $\{(1-zx_1)x_2=0\}$. By removing the strict transform of $\{x_2=0\}$ one recovers the geometry of Figure \ref{fig1}.}\label{fig2}
\end{figure}

Therefore, $\Li_n(z)$ is a period function of the family of relative cohomology groups 
\begin{equation}\label{eqn:simplicialmotive}
  \H^n\left(\widetilde{X}'_n \setminus \widetilde{A}'_n, \left(\widetilde{B}'_n \cup E_n\right) \setminus \widetilde{A}'_n\cap  \left(\widetilde{B}'_n\cup E_n\right)\right).
\end{equation}
Notice that the change of variables \eqref{eqn:changeofvar} provides a local chart for the blow-up $\widetilde{X}'_n$. More precisely, it induces an isomorphism  
\begin{align*}
  X_n \longrightarrow \widetilde{X}'_n \setminus \left(\widetilde{\{x_2=0\}} \cup \cdots \cup \widetilde{\{x_n=0\}}\right), \quad 
    (t_1, \dots, t_n) \longmapsto (x_1, \dots, x_n)
\end{align*} that identifies the pairs $(X_n \setminus A_n, B_n\setminus A_n \cap B_n)$ and $(\widetilde{X}'_n \setminus \widetilde{A}'_n, (\widetilde{B}'_n \cup E_n) \setminus \widetilde{A}'_n\cap  (\widetilde{B}'_n\cup E_n))$. Therefore, the cohomology groups \eqref{eqn:relativecohom} and \eqref{eqn:simplicialmotive} are isomorphic. Before this paper, the latter had only been computed for $n=2$ by Wang \cite{Wang}, who gives a slightly different presentation in the spirit of Goncharov--Manin \cite{goncharovmanin} and proves that it is a motivic lift of the dilogarithm variation. Our definition here can thus be seen as a way to circumvent the blow-up process. 

Another approach to constructing the polylogarithm motive is via a quotient of the motivic fundamental groupoid $\pi_1^{\mot}(\PP^1_\ZZ \setminus \{0, 1, \infty\}; \overline{0}, z)$ with a tangential basepoint at $0$ and a \emph{varying} usual basepoint at $z$, of which the iterated integrals \eqref{eqn:simplicialLin} are naturally periods. At the moment of writing, this motivic fundamental group has been constructed as a system of realizations by Deligne \cite{delignegroupefondamental} and as mixed Tate motive by Deligne--Goncharov \cite{delignegoncharov}, but only for a \emph{fixed algebraic} value of the parameter $z$. Moreover, the construction is not purely motivic, as it relies on the fact that the Hodge realization functor on mixed Tate motives over a number field is fully faithful and its image is stable under subquotients. 

\subsection{Overview} The paper is organized as follows. Section~\ref{sec:1} is devoted to the proof of the isomorphism \eqref{eq: crucial isomorphism}, which is achieved in Theorem~\ref{thm: structureLog}. From this, we derive the structure of the motive $\mathcal{L}_n$ at the beginning of Section~\ref{sec:2}. We then compute its Hodge realization and show that it agrees with the polylogarithm variation. The paper is supplemented by Appendix~\ref{sec: appendix relative cohomology motives}, in which we gather the main properties of relative cohomology motives, Appendix~\ref{sec: appendix postnikov}, which presents a computation of motives of configuration spaces that is used in the proof of Proposition~\ref{prop: iso C D alternating}, and Appendix~\ref{sec: appendix ext}, which gives details on certain extension groups and clarifies the relation with previous approaches to polylogarithm motives.
 
\subsection{Notation and conventions related to categories of motives} 

 Voevodsky's triangulated category of mixed motives with rational coefficients over a scheme $S$ (also known as \emph{motivic sheaves} on $S$) will be denoted by $\DM(S)$. It is a symmetric monoidal category with unit object~$\QQ_S(0)$. The assignment $S\leadsto \DM(S)$ supports a six-functor formalism; see \cite{ayoubPhD1,ayoubPhD2, cisinskideglise}. Every pair of varieties $(X, Z)$ over $S$ gives rise to an object $\M(X,Z)$ of $\DM(S)$ that we call a \emph{relative cohomology motive}. We warn the reader that in the literature such notation is often used for relative \emph{homology} motives instead. Appendix~\ref{sec: appendix relative cohomology motives} contains a compendium of results on these objects.
        
        We let $\GG_{m,S}$ denote the multiplicative group scheme over $S$ and $\{1\}\hookrightarrow \GG_{m,S}$ its unit section. We define the \emph{Lefschetz motive} as the relative cohomology motive 
        \begin{equation}\label{eqn:Lefschetz}
        \QQ_S(-1) = \M\left(\GG_{m,S},\{1\}\right)\,[1]. 
        \end{equation} 
        It is an invertible object of $\DM(S)$, and hence the tensor powers $\QQ_S(-n)=\QQ_S(-1)^{\otimes n}$ are defined for each~$n\in \ZZ$. Given an object $\mathcal{F}$ of $\DM(S)$, we set $\mathcal{F}(-n)=\mathcal{F}\otimes \QQ_S(-n)$.

        Let $\DMT(S)$ denote the triangulated subcategory of $\DM(S)$ generated by the objects $\QQ_S(-n)$ for all~$n\in \ZZ$. By a theorem of Levine \cite{levinetatemotives}, if $S$ satisfies the Beilinson--Soul\'{e} vanishing conjecture, then $\DMT(S)$ is equipped with a canonical $\t$-structure, whose heart $\MT(S)$ is called the category of \emph{mixed Tate motives} with rational coefficients over~$S$. It is an abelian category, endowed with cohomology functors induced by the~$\t$\nobreakdash-structure $\H^n \colon \DMT(S)\rightarrow \MT(S)$.  Every object $\mathcal{F}$ of~$\MT(S)$ is equipped with an increasing \emph{weight filtration} $W$ indexed by even integers such that, for each $n\in\ZZ$, the graded piece 
        \[
        \Gr_{2n}^W\mathcal{F} = W_{2n}\mathcal{F}/W_{2(n-1)}\mathcal{F}
        \] is a finite direct sum of copies of $\QQ_S(-n)$. We call \emph{semisimplification} of $\mathcal{F}$ the semisimple object 
        \[
        \Gr^W\mathcal{F}=\bigoplus_{n \in \ZZ} \Gr_{2n}^W\mathcal{F}
        \] of the abelian category~$\MT(S)$. Moreover, the symmetric monoidal structure on $\DM(S)$ induces the structure of a $\QQ$-linear neutral Tannakian category on~$\MT(S)$. 

        By the work of Borel, $\Spec(\QQ)$ satisfies the Beilinson--Soul\'{e} vanishing conjecture. Since the map $K_i(\ZZ)_\QQ\to K_i(\QQ)_\QQ$ is an injection for $i=1$ and an isomorphism for $i\neq 1$, so does $\Spec(\ZZ)$. Using the homotopy invariance $K_i(\mathbb{Z}[z])\simeq K_i(\mathbb{Z})$ and the  localization long exact sequence as in \hbox{\cite[Corollary 6.6.2]{esnaultlevinetatemotives}}, we see that $S=\PP^1_\ZZ\setminus\{0,1,\infty\}=\AA^1_\ZZ\setminus\{0,1\}$ also satisfies the Beilinson--Soul\'{e} vanishing conjecture, and therefore we can work in the Tannakian category~$\MT(S)$.

\subsection*{Acknowledgments} We would like to thank Tanguy Rivoal for bringing Deligne's letter to Beilinson \cite{DeltoBeil} to our attention, and Simon Pepin Lehalleur for his many comments on a first version of this paper. 

\section{A geometric description of the logarithmic system}\label{sec:1}

The main result of this section is Theorem~\ref{thm: structureLog}, which identifies the symmetric powers of the Kummer motive with explicit relative cohomology motives. Throughout, we work over the base scheme \hbox{$S=\PP^1_\ZZ\setminus \{0,1,\infty\}$} with coordinate~$z$, and we let $\GG_{m,S}$ denote the multiplicative group scheme over $S$ with coordinate $t$.

     \subsection{The Kummer motive and the logarithmic system}
    
        Let $Z \subset \GG_{m,S}$ denote the union of the closed \hbox{$S$\nobreakdash-subschemes} of $\GG_{m,S}$ defined by the equations $\{t=1\}$ and~$\{tz=1\}$.
        
        \begin{definition}
        The \emph{Kummer motive} is the relative cohomology motive
        $$\mathcal{K} = \M\left(\GG_{m,S},Z\right)\,[1] \;\;\in \DM(S).$$
        \end{definition}
        
        The Kummer motive is usually defined over the base $\PP^1_\ZZ\setminus \{0,\infty\}=\GG_{m,\ZZ}$, and $\mathcal{K}$ is its restriction to $S$. As the following classical proposition shows, it is an extension of $\QQ_S(-1)$ by $\QQ_S(0)$. The fiber at $z$ of its \'{e}tale realization is the Kummer torsor of roots of $z$, whence its name; see \cite[Sections~2.9 and~2.10]{delignegroupefondamental}.
        
        \begin{proposition}\label{prop: kummer}
        The Kummer motive $\mathcal{K}$ belongs to the subcategory $\MT(S)$ and fits into a short exact sequence
        \begin{equation}\label{eq: short exact sequence kummer}0 \longrightarrow \QQ_S(0) \longrightarrow \mathcal{K} \longrightarrow \QQ_S(-1) \longrightarrow 0.
        \end{equation}
        \end{proposition}
        
        \begin{proof}
        Applied to the closed subschemes $Z'=\{t=1\}$ and $Y=\{tz=1\}$ of $X=\GG_{m,S}$, Proposition~\ref{prop: appendix boundary triangle} from the appendix yields the distinguished triangle 
        \begin{displaymath}
        \M(Y) \longrightarrow \mathcal{K} \longrightarrow \M(X, Z')\,[1]\stackrel{+1}{\longrightarrow}
        \end{displaymath} in the category $\DM(S)$. The rightmost term is $\QQ_S(-1)$ by definition \eqref{eqn:Lefschetz}, and the leftmost term is isomorphic to~$\QQ_S(0)$ since $p_Y\colon Y\rightarrow S$ is an isomorphism. As these two objects belong to the subcategory $\MT(S)$, the result follows.  
        \end{proof}
            
        Let $n \geq 0$ be an integer. We will be interested in the $\supth{n}$ symmetric power of the Kummer motive. By Proposition~\ref{prop: kummer}, its semisimplification is equal to
        $$\Gr^W\Sym^n(\mathcal{K}) \simeq \Sym^n(\Gr^W\mathcal{K}) \simeq \Sym^n(\QQ_S(0)\oplus\QQ_S(-1))\simeq  \QQ_S(0)\oplus \QQ_S(-1)\oplus\cdots\oplus \QQ_S(-n).$$
        The inclusion $\iota\colon\QQ_S(0)\hookrightarrow \mathcal{K}$ appearing in the short exact sequence \eqref{eq: short exact sequence kummer} induces, for each $n \geq 1$, transition morphisms $\iota_n\colon\Sym^{n-1}(\mathcal{K})\hookrightarrow \Sym^n(\mathcal{K})$ given by the formula
        \begin{equation}\label{eq: transition morphism Sym K}
        \iota_n=\frac{1}{n}\sum_{i=1}^{n} \id^{\otimes (i-1)}\otimes \iota \otimes \id^{\otimes (n-i)},
        \end{equation}
        which make up an inductive system $\Sym(\mathcal{K})$ in the category $\MT(S)$. 
        
        \begin{definition}
        The inductive system $\Sym(\mathcal{K})$ is called the \emph{logarithmic system}.
        \end{definition}
        
    \subsection{The inductive system $\mathcal{T}$}\label{subsec: definition inductive system T}
    
        For each integer $n \geq 0$, let $\AA^{n+1}_S$ denote the $(n+1)$-dimensional affine space over $S$ with coordinates $(t_1,\ldots,t_{n+1})$, and let $T_n\subset \AA^{n+1}_S$ be the closed $S$-subscheme defined by the equation $\{z t_1\cdots t_{n+1}=1\}$. Let~$Z_n\subset T_n$ be the union of the closed $S$-subschemes \hbox{$Z_n^i=\{t_i=1\}$} for $i=1,\ldots, n+1$. Under the identification $T_n\simeq \GG_{m,S}^n$, with coordinates $(t_1,\ldots,t_n)$, the subscheme $Z_n^i$ corresponds to the subtorus given by $\{t_i=1\}$ for~$i=1,\ldots,n$ and $\{zt_1\cdots t_n=1\}$ for $i=n+1$. We define an object
        \[
        \mathcal{T}_n = \M(T_n,Z_n)\,[n] \;\; \in \DM(S). 
        \]
       Since $(T_0, Z_0) \simeq (S, \emptyset)$ and $(T_1, Z_1) \simeq (\GG_{m, S}, \{t=1\} \cup \{tz=1\})$, we have $\mathcal{T}_0\simeq\QQ_S(0)$ and $\mathcal{T}_1\simeq\mathcal{K}$.
        
        We let $Z'_n\subset T_n$ denote the union of the subtori $Z_n^i$ for $i=1,\ldots,n$, so that $Z_n=Z'_n\cup Z_n^{n+1}$. For each $n \geq 1$, the pair $(Z_n^{n+1},Z_n^{n+1}\cap Z'_n)$ is naturally identified with the pair $(T_{n-1}, Z_{n-1})$, so we have a partial boundary morphism (see Definition~\ref{defi: partial boundary morphism}) along $Z_n^{n+1}$ that we denote by 
        \[
        \tau_n\colon \mathcal{T}_{n-1} \longrightarrow \mathcal{T}_n.
        \]
        These morphisms give rise to an inductive system $\mathcal{T}$.
        
        \begin{proposition}\label{prop: induction T}
        The object $\mathcal{T}_n$ belongs to the category $\MT(S)$, and the morphism $\tau_n$ fits into a short exact sequence
        \[
        0\longrightarrow \mathcal{T}_{n-1} \stackrel{\tau_n}{\longrightarrow} \mathcal{T}_n \longrightarrow \QQ_S(-n) \longrightarrow 0.
        \]
        \end{proposition}
        
        \begin{proof}
        By Proposition~\ref{prop: appendix boundary triangle}, the morphism $\tau_n$ fits into a distinguished triangle 
        \[
        \mathcal{T}_{n-1} \stackrel{\tau_n}{\longrightarrow} \mathcal{T}_n \longrightarrow \M(T_n,Z'_n)\,[n] \stackrel{+1}{\longrightarrow}\
        \]
      in the category $\DM(S)$. Since the pair $(T_n,Z'_n)$ is the $\supth{n}$ Cartesian power of the pair $(T_1,Z'_1)$, the K\"unneth formula (see Proposition~\ref{prop: appendix kunneth}) gives an isomorphism 
      \[
      \M(T_n,Z'_n)[n]\simeq \left(\M(T_1,Z'_1)[1]\right)^{\otimes n}=\QQ_S(-n), 
      \]
      taking into account the equality $\M(T_1,Z'_1)[1]=\QQ_S(-1)$, which is the definition of the right-hand side. Starting with $\mathcal{T}_0\simeq \QQ_S(0)\in \MT(S)$, the statement then follows by induction on~$n$. 
        \end{proof}
        
        We note that Proposition~\ref{prop: induction T} and induction on $n$, starting with $\mathcal{T}_0\simeq \QQ_S(0)$, imply that the semisimplification of $\mathcal{T}_n$ is 
        $$\Gr^W\mathcal{T}_n \simeq \QQ_S(0)\oplus \QQ_S(-1) \oplus \cdots \oplus \QQ_S(-n).$$
        Therefore, $\Sym^n(\mathcal{K})$ and $\mathcal{T}_n$ have the same semisimplification. The aim of the next four subsections is to prove that these two motives are actually isomorphic for each $n \geq 0$, in a way compatible with the inductive systems: $\Sym(\mathcal{K})\simeq \mathcal{T}$ (see Theorem~\ref{thm: structureLog} below). After introducing auxiliary inductive systems $\mathcal{C}$ and $\mathcal{D}$ in Section~\ref{subsec: auxiliary C D}, this is done in three steps (achieved in Sections~\ref{subsec: first iso},~\ref{subsec: second iso},~\ref{subsec: third iso}) 
        
        \begin{remark}\label{rem: symmetric group action T}
        The symmetric group $\mathfrak{S}_{n+1}$ acts on $\mathbb{A}^{n+1}_S$ by permuting the coordinates $t_i$, and this action preserves the subschemes $T_n$ and $Z_n$. By the functoriality of relative cohomology motives (see Proposition~\ref{prop: appendix relative cohomology motives functorial}), $\mathfrak{S}_{n+1}$ thus acts on the object $\mathcal{T}_n$ of $\MT(S)$, and one can prove (by induction on~$n$) that it does so via the alternating character $\sgn_{n+1}$. We omit the proof of this fact, which will not be used in the rest of the article.
        \end{remark}

    \subsection{The auxiliary inductive systems $\boldsymbol{\mathcal{C}}$ and $\boldsymbol{\mathcal{D}}$}\label{subsec: auxiliary C D}
    
        For the remainder of this section, we make a change of coordinates in $T_n\simeq \GG_{m,S}^n$ by setting
        $$(x_1,x_2,\ldots,x_n)=(t_1,t_1t_2,\ldots,t_1t_2\cdots t_n),$$ 
        so that $Z_n$ is the union of the subvarieties $\{x_1=1\}$, $\{x_n=1/z\}$, and $\{x_i=x_{i+1}\}$ for $i=1,\ldots,n-1$. We define closed $S$-subschemes $C_n\subset D_n\subset T_n$ by 
        $$
        C_n=\bigcup_{1\leq i\leq n}\{x_i=1\}\cup\{x_i=1/z\},\quad\Delta=\bigcup_{1\leq i<j\leq n}\{x_i=x_j\},\quad\text{and}\quad D_n=C_n\cup \Delta.
        $$ Note that $Z_n$ is contained in $D_n$, and the intersection $Z_n \cap C_n$ consists of $\{x_1=1\} \cup \{x_n=1/z\}$. 
        We define objects of $\DM(S)$:
        $$\mathcal{C}_n=\M(T_n,C_n)\,[n] \quad\text{and}\quad \mathcal{D}_n=\M(T_n,D_n)\,[n].$$
        By the functoriality of relative cohomology motives (see Proposition~\ref{prop: appendix relative cohomology motives functorial}), there is a morphism
        \[\varphi_n\colon\mathcal{D}_n\longrightarrow\mathcal{C}_n.\]
        
        The symmetric group $\mathfrak{S}_n$ acts on $T_n$ by permuting the coordinates $x_i$, and this action preserves the closed subschemes $C_n$ and $D_n$. Again by the functoriality of relative cohomology motives, we thus have an action of $\mathfrak{S}_n$ on the objects $\mathcal{C}_n$ and $\mathcal{D}_n$ of $\DM(S)$, which is such that the morphism~$\varphi_n$ is $\mathfrak{S}_n$\nobreakdash-equivariant. (Note that this symmetric group action has nothing to do with the action of~$\mathfrak{S}_{n+1}$ on $T_n$ discussed in Remark~\ref{rem: symmetric group action T}.) Since~$\DM(S)$ is a pseudo-abelian category, we can speak about the alternating components $\mathcal{C}_n^\sgn$ and $\mathcal{D}_n^\sgn$, which are direct summands of~$\mathcal{C}_n$ and~$\mathcal{D}_n$, respectively, and we get an induced morphism $\varphi_n^\sgn\colon\mathcal{D}_n^\sgn\rightarrow \mathcal{C}_n^\sgn$. (We will see in the next subsection that $\mathcal{C}_n$ and $\mathcal{D}_n$ actually live in the abelian subcategory $\MT(S)$ of $\DM(S)$.)  
         
        For $i=1,\ldots,n$, consider 
        $$
        Y(i)=\{x_i=1/z\} \quad \text{and}\quad Z'(i)=\{x_i=1\}\cup\bigcup_{j\neq i}\left(\{x_j=1\}\cup\{x_j=1/z\}\right),$$ so that $C_n=Z'(i)\cup Y(i)$. There is a natural identification of the pair $(Y(i),Y(i)\cap Z'(i))$ with $(T_{n-1},C_{n-1})$, and hence a partial boundary morphism (see Definition~\ref{defi: partial boundary morphism}) along $Y(i)$ that we denote by $\gamma_n^i\colon\mathcal{C}_{n-1}\rightarrow\mathcal{C}_n$. We consider the morphism $\gamma_n\colon\mathcal{C}_{n-1}\rightarrow \mathcal{C}_n$ defined by the formula
        \begin{equation}\label{eq: transition morphism C and D}
        \gamma_n=\frac{1}{n}\sum_{i=1}^n(-1)^{n-i}\gamma^i_n.
        \end{equation}
        The collection of the motives $\mathcal{C}_n$, endowed with the transition morphisms $\gamma_n$, makes up an inductive system~$\mathcal{C}$ in $\DM(S)$. The next lemma shows that this passes to the alternating components and gives rise to a direct summand $\mathcal{C}^\sgn$ of the inductive system $\mathcal{C}$.

        \begin{lemma}\label{lem: C sgn inductive system}
        The transition morphism $\gamma_n$ induces a morphism $\gamma_n^\sgn\colon\mathcal{C}_{n-1}^\sgn\rightarrow\mathcal{C}_n^\sgn$.
        \end{lemma}
        
        \begin{proof}
        We have the following identity in the group algebra $\QQ[\mathfrak{S}_n]$. Let $\pi_n = \frac{1}{n!}\sum_{\sigma\in\mathfrak{S}_n}\sgn(\sigma)\sigma$ denote the projector onto the alternating component, and let us define
        $$\xi_n=\frac{1}{n}\sum_{i=1}^n(-1)^{n-i}(i\cdots n).$$
        Then we have $\pi_n=\xi_n\pi_{n-1}$, which follows immediately from the fact that a permutation $\sigma\in\mathfrak{S}_n$ can be uniquely written as a product $(i\cdots n)\sigma'$ with $i\in\{1,\ldots,n\}$ and $\sigma'\in\mathfrak{S}_{n-1}$. 
        The functoriality of partial boundary morphisms (see Proposition~\ref{prop: appendix functoriality boundary}) with respect to $(i\cdots n)$ implies that we have $\gamma_n^i=(i\cdots n)\circ \gamma_n^n$, so that we may write $\gamma_n=\xi_n\circ \gamma_n^n$. Again by the functoriality of partial boundary morphisms, $\gamma_n^n$ is~$\mathfrak{S}_{n-1}$\nobreakdash-equivariant, and we have
        $$\gamma_n\circ \pi_{n-1} = \xi_n\circ \gamma_n^n\circ \pi_{n-1} = (\xi_n\pi_{n-1})\circ \gamma_n^n = \pi_n\circ \gamma_n^n\ ,$$
        and thus $\gamma_n$ sends the image of $\pi_{n-1}$ to the image of $\pi_n$. The claim follows from this. 
        \end{proof} 
        
        In the same fashion, we have morphisms $\delta_n^i\colon\mathcal{D}_{n-1}\rightarrow\mathcal{D}_n$, which are partial boundary morphisms along~$\{x_i=1/z\}$, and $\delta_n\colon\mathcal{D}_{n-1}\rightarrow \mathcal{D}_n$ given by the formula
        $$\delta_n=\frac{1}{n}\sum_{i=1}^n(-1)^{n-i}\delta^i_n,$$
        which define an inductive system $\mathcal{D}$. By the same argument as in the proof of Lemma~\ref{lem: C sgn inductive system}, we see that the alternating components $\mathcal{D}_n^\sgn$ make up a direct summand inductive system $\mathcal{D}^\sgn$. 
        
        By the functoriality of partial boundary morphisms (see Proposition~\ref{prop: appendix functoriality boundary}), $\varphi_n\circ\gamma_n^i=\delta_n^i\circ\varphi_{n-1}$ holds for all $n$ and $i$, and we get morphisms of inductive systems $\varphi\colon\mathcal{D}\rightarrow\mathcal{C}$ and $\varphi^\sgn\colon\mathcal{D}^\sgn\rightarrow\mathcal{C}^\sgn$.

    \subsection{A first isomorphism: $\boldsymbol{\Sym(\mathcal{K})\simeq \mathcal{C}^\sgn}$}\label{subsec: first iso}

        \begin{proposition}\label{prop: C Sym Kummer}
        The object $\mathcal{C}_n$ lives in the category $\MT(S)$, and we have an $\mathfrak{S}_n$-equivariant isomorphism
        $$ \mathcal{K}^{\otimes n} \stackrel{\lowsimeq}{\longrightarrow} \mathcal{C}_n\otimes\sgn_n,$$
        which induces an isomorphism of inductive systems
        $$\Sym(\mathcal{K})\stackrel{\lowsimeq}{\longrightarrow}\mathcal{C}^{\sgn}.$$
        \end{proposition}
        
        \begin{proof}
        The pair $(T_n,C_n)$ is the $\supth{n}$ Cartesian power of the pair $(T_1,C_1)$, and hence by the K\"{u}nneth formula (see Proposition~\ref{prop: appendix kunneth}), there is an isomorphism $\mathcal{K}^{\otimes n} =\left(\mathcal{C}_1\right)^{\otimes n}\stackrel{\lowsimeq}{\rightarrow}\mathcal{C}_n$ in $\DM(S)$. Because of the shift $[1]$ in the definition of $\mathcal{K}$, this isomorphism is $\mathfrak{S}_n$-equivariant up to the alternating character $\sgn_n$ (Koszul sign rule), whence the first claim. We therefore get an isomorphism
        $$\Sym^n(\mathcal{K}) = \left(\mathcal{K}^{\otimes n}\right)^{\mathfrak{S}_n} \simeq (\mathcal{C}_n\otimes \sgn_n)^{\mathfrak{S}_n} \simeq \mathcal{C}_n^{\sgn}.$$
        From the compatibility between partial boundary morphisms and the K\"{u}nneth formula (see Proposition~\ref{prop: appendix kunneth boundary}), we get the following commutative diagram for every $i=1,\ldots,n$, where the sign comes from the Koszul sign~rule:
        $$\diagram{
        \mathcal{K}^{\otimes n} \ar[r]^-{\simeq}& \mathcal{C}_n \\
        \mathcal{K}^{\otimes n-1} \ar[r]^-{\simeq} \ar[u]^-{\id^{\otimes (i-1)}\otimes \iota \otimes \id^{\otimes (n-i)}}& \mathcal{C}_{n-1} \ar[u]_-{(-1)^{n-i}\gamma_n^i}\rlap{.}
        }$$
        This implies, by looking at formulas \eqref{eq: transition morphism Sym K} and \eqref{eq: transition morphism C and D}, that the isomorphisms $\Sym^n(\mathcal{K})\simeq \mathcal{C}_n^\sgn$ are compatible with the inductive systems.
        \end{proof}
            
    \subsection{A second isomorphism: $\boldsymbol{\mathcal{C}^\sgn\simeq\mathcal{D}^\sgn}$}\label{subsec: second iso}
       
        \begin{proposition}\label{prop: iso C D alternating}
        The objects $\mathcal{D}_n$ live in the category $\MT(S)$. The morphism $\varphi\colon\mathcal{D}\rightarrow\mathcal{C}$ induces an isomorphism of inductive systems
$$\varphi^\sgn\colon\mathcal{D}^\sgn\stackrel{\lowsimeq}{\longrightarrow} \mathcal{C}^\sgn.$$
        \end{proposition}
        
        \begin{proof}
        Recall the notation $\Delta=\bigcup_{1\leq i<j\leq n}\{x_i=x_j\}$ for the fat diagonal in $T_n$. We define
        $$\mathcal{F} = \left(j_{T_n\setminus C_n}^{T_n}\right)_!\QQ_{T_n\setminus C_n}(0)[n]\;\in \DM(T_n),$$
        so that we have $\mathcal{C}_n=\left(p_{T_n}\right)_*\mathcal{F}\ .$
        By using $\QQ_{T_n\setminus D_n}(0)\simeq \left(j_{T_n\setminus D_n}^{T_n\setminus C_n}\right)^*\QQ_{T_n\setminus C_n}(0)$ and the base change isomorphism $\left(j_{T_n\setminus D_n}^{T_n\setminus \Delta}\right)_!\left(j_{T_n\setminus D_n}^{T_n\setminus C_n}\right)^*\simeq \left(j_{T_n\setminus \Delta}^{T_n}\right)^*\left(j_{T_n\setminus C_n}^{T_n}\right)_!,$ we get an isomorphism
        $$\mathcal{D}_n\simeq \left(p_{T_n}\right)_*\left(j_{T_n\setminus \Delta}^{T_n}\right)_!\left(j_{T_n\setminus \Delta}^{T_n}\right)^*\mathcal{F}.$$
        We now use the notation and results of Appendix~\ref{sec: appendix postnikov}. We apply Theorem~\ref{thm: postnikov system} to the motive $\mathcal{F}$ on $T_n=\GG_{m,S}^n$ and apply the functor $\left(p_{T_n}\right)_*$ to it. We obtain a Postnikov system in $\DM(S)$:
        $$\diagram{
	    0=F^n\ar[rr]&&\cdots\ar[ld]& \cdots\ar[rr] && F^1\ar[rr]\ar[ld] && F^0=\mathcal{D}_n,\ar[ld] \\
	    &G^{n-1}\ar[ul]^{+1}& && G^1\ar[ul]^{+1} && G^0 = \mathcal{C}_n\ar[ul]^{+1} &
	    }$$
        whose graded objects are 
        $$G^k= \bigoplus_{\substack{\pi\in\Pi_n \\ |\pi|=n-k}}  \left(p_{T_\pi}\right)_*\left(i_{T_\pi}^{T_n}\right)^*\mathcal{F}[-k] \;\otimes\; A(\pi)^\vee,$$
        where $T_\pi=\GG_{m,S}^\pi$ denotes the closed subscheme of $T_n$ where $x_a=x_b$ if $a$ and $b$ are in the same block of $\pi$. We have a base change isomorphism 
$$\left(p_{T_\pi}\right)_*\left(i_{T_\pi}^{T_n}\right)^*\mathcal{F}[-k]\simeq \left(p_{T_\pi}\right)_*\left(j_{T_\pi\setminus T_\pi\cap C_n}^{T_
        \pi}\right)_!\QQ_{T_\pi\setminus T_\pi\cap C_n}(0)[n-k]\simeq \mathcal{C}_\pi,$$
        where $\mathcal{C}_\pi\simeq \mathcal{C}_{n-k}$ is defined in the same way as $\mathcal{C}_n$ on the torus $T_\pi\simeq T_{n-k}$. This way we can write
        $$G^k \simeq \bigoplus_{\substack{\pi\in\Pi_n \\ |\pi|=n-k}} \mathcal{C}_\pi \otimes A(\pi)^\vee,$$
        and by Proposition~\ref{prop: C Sym Kummer}, the object $G^k$ lives in the category $\MT(S)$ for all $k$. Thus, this is also the case for the objects $F^k$ and in particular for $F^0=\mathcal{D}_n$, which proves the first claim. We thus get a descending filtration $F^k$ on 
        $\mathcal{D}_n$ whose graded quotients are the objects $G^k$. We note that this filtration is $\mathfrak{S}_n$-equivariant by Theorem~\ref{thm: postnikov system} and that the last quotient $F^0=\mathcal{D}_n\twoheadrightarrow G^0\simeq \mathcal{C}_n$ is nothing but the morphism $\varphi_n$ by construction.
        
        We now prove the equality $(G^k)^\sgn=0$ for all $k>0$. The symmetric group $\mathfrak{S}_n$ permutes the summands of $G^k$. The stabilizer of the summand indexed by a partition $\pi$ is the subgroup \hbox{$\mathfrak{S}(\pi)=\prod_{B\in \pi}\mathfrak{S}_B$}; it acts trivially on $\mathcal{C}_\pi$ (because it acts trivially on the torus $T_\pi$), and its action on $A(\pi)^\vee=\bigotimes_{B\in\pi}A_B^\vee$ is induced by the action of $\mathfrak{S}_B$ on $A_B^\vee$ for each block $B$ of $\pi$. We therefore have an inclusion
        $$(G^k)^\sgn \longhookrightarrow \bigoplus_{\substack{\pi\in \Pi_n\\ |\pi|=n-k}} \left(\mathcal{C}_\pi \otimes \bigotimes_{B\in\pi} (A_B^\vee)^\sgn\right).$$
        If $k=n-|\pi|>0$, then there is a block $B$ of $\pi$ of cardinality at least $2$, for which $(A_B^\vee)^\sgn=0$, by Theorem~\ref{thm: appendix Arnold sign}. Therefore, $(G^k)^\sgn=0$ for $k>0$. This implies by backward induction the equality $(F^k)^\sgn=0$ for $k>0$, and in particular $(F^1)^\sgn=0$. Now the short exact sequence
        $$0\longrightarrow F^1\longrightarrow \mathcal{D}_n\stackrel{\varphi_n}{\longrightarrow} \mathcal{C}_n\longrightarrow 0$$
        implies that $\varphi_n^\sgn$ is an isomorphism. 
        \end{proof}
        
        \begin{remark}\label{rem: topological interpretation iso C D alternating}
        In weight zero, the isomorphism of Proposition~\ref{prop: iso C D alternating} is a combinatorial statement that can be understood as follows in the Betti realization. Let us consider the fiber at a fixed \hbox{$z\in S(\CC)=\CC\setminus\{0,1\}$}, which for simplicity we assume to be a real number satisfying $0<z<1$. The transpose of the Betti realization of~$\varphi_n$ is the natural map
        \begin{equation}\label{eq: Betti C D}
        \H_n^{\B}(T_n,C_n)\longrightarrow \H_n^{\B}(T_n,D_n)
        \end{equation}
        which sends the class of a cycle on $T_n$ with boundary along $C_n$ to the class of the same cycle, viewed with a boundary along $D_n\supset C_n$. 
        Then the weight zero quotient of~$\H_n^{\B}(T_n,C_n)$ is $1$\nobreakdash-dimensional with basis the class of the hypercube $\square_n(z) = [1,1/z]^n$. The weight zero quotient of~$\H_n^{\B}(T_n,D_n)$ has dimension~$n!$ and has a basis consisting of the classes of the simplices 
        \begin{displaymath}
        \Delta_n^\sigma(z)=\left\{(x_1,\ldots,x_n)\in \mathbb{R}^n \; |\; 1\leq x_{\sigma^{-1}(1)}\leq \cdots \leq x_{\sigma^{-1}(n)} \leq 1/z\right\}
        \end{displaymath} for $\sigma\in \mathfrak{S}_n$. If we take the action of the symmetric group into account, then the weight zero quotient of the map \eqref{eq: Betti C D} is the natural inclusion
        $$\sgn_n \longrightarrow \sgn_n\otimes \QQ[\mathfrak{S}_n]$$
        sending the class of the hypercube $\square_n(z)$ to the alternating sum of the classes of the simplices $\Delta_n^\sigma(z)$, which reflects the fact that $\square_n(z)$ is paved by the $\Delta_n^\sigma(z)$ for $\sigma\in\mathfrak{S}_n$ (the signs are forced by the orientations). This last map clearly induces an isomorphism on the alternating components.
        \end{remark}
        
    \subsection{A third isomorphism: $\boldsymbol{\mathcal{D}^\sgn\simeq \mathcal{T}}$}\label{subsec: third iso}
    
        By the functoriality of relative cohomology motives, there is a natural morphism $\widetilde{\psi}_n\colon\mathcal{D}_n\rightarrow \mathcal{T}_n$. We let 
        $$\psi_n\colon\mathcal{D}_n\longrightarrow\mathcal{T}_n$$
        denote this morphism rescaled by the factor $n!$; \textit{i.e.}, $\psi_n=n!\; \widetilde{\psi}_n$. The next lemma shows that these morphisms induce a morphism of inductive systems $\psi\colon\mathcal{D}\rightarrow \mathcal{T}$.
            
        \begin{lemma}
        We have $\psi_n\circ \delta_n=\tau_n\circ \psi_{n-1}$.
        \end{lemma}
        
        \begin{proof}
        By the functoriality of boundary morphisms (see Proposition~\ref{prop: appendix functoriality boundary}), we have $\widetilde{\psi}_n\circ \delta_n^i=0$ for all~\hbox{$i=1,\ldots,n-1$}, and $\widetilde{\psi}_n\circ \delta_n^n=\tau_n\circ \widetilde{\psi}_{n-1}$. The claim follows since $n!\times (1/n) = (n-1)!\ $.
        \end{proof}
        
        \begin{proposition}\label{prop: third iso}
        The composite
        \begin{equation}\label{eq: composite third iso}
        \mathcal{D}^\sgn\longhookrightarrow \mathcal{D}\stackrel{\psi}{\longrightarrow} \mathcal{T}
        \end{equation}
        is an isomorphism of inductive systems.
        \end{proposition}
        
        \begin{proof}
        We show by induction on $n\geq 0$ that the composite $\mathcal{D}_n^\sgn\hookrightarrow \mathcal{D}_n\stackrel{\psi}{\longrightarrow} \mathcal{T}_n$ is an isomorphism. The case~$n=0$ is trivial since $\mathcal{D}_0^\sgn= \mathcal{D}_0=\mathcal{T}_0=\QQ_S(0)$. For the inductive step, we consider the following commutative diagram in the category $\MT(S)$:
        $$\diagram{
        0 \ar[d] & & 0 \ar[d]\\
        \mathcal{D}_{n-1}^\sgn \ar[d]_{\delta_n^\sgn} \ar@{^(->}[r] & \mathcal{D}_{n-1} \ar[d]_{\delta_n}\ar[r]^{\psi_{n-1}} & \mathcal{T}_{n-1} \ar[d]^{\tau_n}\\
        \mathcal{D}_n^\sgn \ar[d]\ar@{^(->}[r]& \mathcal{D}_n \ar[r]_{\psi_n} & \mathcal{T}_n \ar[d] \\
        \QQ_S(-n) \ar[d] && \QQ_S(-n) \ar[d]\\
        0 & & 0\rlap{.}
        }$$
        The first column is exact because of Propositions~\ref{prop: C Sym Kummer} and~\ref{prop: iso C D alternating}. The third column is exact because of Proposition~\ref{prop: induction T}. Using the induction hypothesis and the five lemma, it is enough to prove that the composite
        $$\QQ_S(-n) \simeq \Gr^W_{2n}\mathcal{D}_n^\sgn \longhookrightarrow \Gr^W_{2n}\mathcal{D}_n\xrightarrow{\Gr^W_{2n}\psi_n}\Gr^W_{2n}\mathcal{T}_n\simeq \QQ_S(-n)$$
        is an isomorphism. It is enough to prove that the morphism $\Gr^W_{2n}\psi_n$ is an isomorphism. By the functoriality of relative cohomology motives, we have a commutative diagram in $\MT(S)$, obtained by taking $\H^0$ of the obvious morphisms in $\DM(S)$ (the subschemes $Z'_n\subset T_n$ were defined in Section~\ref{subsec: definition inductive system T}):
        $$\diagram{
        \mathcal{D}_n \ar[rr]^{\widetilde{\psi}_n} \ar[d]&&\mathcal{T}_n \ar[d]\\
        \mathcal{C}_n \ar[r]& \H^n\left(\M(T_n)\right)& \M(T_n,Z'_n)\,[n]\rlap{.}\ar[l]\\
        }$$
        We now take $\Gr_{2n}^W$ everywhere in this diagram and conclude that $\Gr_{2n}^W\widetilde{\psi}_n$ is an isomorphism because the other four arrows are. Indeed, the five objects appearing in the diagram are relative cohomology motives of the form $\M(T_n,Y)$, where $Y$ is a strict closed subvariety of $T_n$, and therefore have their $\Gr_{2n}^W$ isomorphic to~$\QQ(-n)$. This concludes the proof of the inductive step, and hence of the proposition.
        \end{proof}
        
        \begin{remark}
        Following up on Remark~\ref{rem: topological interpretation iso C D alternating}, let us consider the transpose of the Betti realization of the composite~\eqref{eq: composite third iso}: 
        \begin{equation}\label{eq: composite third iso Betti}
        \H_n^{\B}(T_n,Z_n)\longrightarrow \H_n^{\B}(T_n,D_n) \longtwoheadrightarrow \H_n^{\B}(T_n,D_n)^\sgn \stackrel{\;\lowsimeq}{\longleftarrow} \H_n^{\B}(T_n,C_n)^\sgn,
        \end{equation}
        where the first map is the natural one multiplied by $n!$.
        Assume that $z$ is a real number satisfying $0<z<1$. Then the weight zero quotient of $\H_n^{\B}(T_n,Z_n)$ is $1$-dimensional with basis the class of the simplex $\Delta^{\id}_n(z)$. Its image by \eqref{eq: composite third iso Betti} is 
        $$n!\,\left[\Delta^{\id}_n(z)\right] = \sum_{\sigma\in\mathfrak{S}_n}\sgn(\sigma)[\Delta_n^\sigma(z)]=[\square_n(z)].$$ 
        This explains the factor $n!$ in the definition of $\psi_n$.
        \end{remark}
        
    \subsection{A geometric description of the logarithmic system}
     
        \begin{theorem}\label{thm: structureLog}
        There is an isomorphism of inductive systems in $\MT(S)$ 
        $$\Sym(\mathcal{K})\stackrel{\lowsimeq}{\longrightarrow} \mathcal{T}.$$
        \end{theorem}
        
        \begin{proof}
        The isomorphism is obtained by composing the three isomorphisms given by Propositions~\ref{prop: C Sym Kummer},~\ref{prop: iso C D alternating}, and~\ref{prop: third iso}.
        \end{proof}
        
        \begin{remark}
        A theorem of Beilinson (see \cite[Theorem 4.1]{goncharovmultiplepolylogs} and \cite[Proposition 3.4]{delignegoncharov}) identifies the dual of the Betti realization $\mathcal{T}_\B$ with the Malcev completion of the $\pi_1(\CC^\times, 1)$-torsor of paths from $1$ to $1/z$ in~$\CC^\times$. (Note that multiplication by $z$ and reversal of paths induce an isomorphism between that torsor and the fundamental path torsor based at $1$ and $z$.) More precisely, there is a tower of isomorphisms
        \begin{equation}\label{eq: iso beilinson}
        \QQ[\pi_1(\CC^\times,1,1/z)]/I^{n+1} \stackrel{\lowsimeq}{\longrightarrow} \H_n^{\B}(T_n,Z_n),
        \end{equation}
        where $I$ is the image of the augmentation ideal of the group algebra $\QQ[\pi_1(\CC^\times,1)]$. 
        Beilinson's isomorphism~\eqref{eq: iso beilinson} sends the class of a continuous path $\gamma$ from $1$ to $1/z$ in $\CC^\times$ to the class of the~$n$\nobreakdash-simplex $\Delta_n(\gamma)=\gamma^n\left(\{0\leq x_1\leq \cdots \leq x_n\leq 1\}\right)$. (If $z$ is a real number satisfying $0<z<1$ and~$\gamma$ is the straight line from $1$ to $1/z$, then $\Delta_n(\gamma)=\Delta_n^{\id}(z)$ in the notation of Remark~\ref{rem: topological interpretation iso C D alternating}.)
        It is now known \cite{delignegoncharov} that the Malcev completion of $\pi_1(\CC^\times, 1,\bullet)$ is (dual to) the Betti realization of an ind-object of $\MT(S)$, the affine ring of the \emph{motivic fundamental group} of $\mathbb{G}_m$, which is isomorphic to $\Sym(\mathcal{K})$ in Deligne's category of systems of realizations; see \cite[Proposition~14.2]{delignegroupefondamental}. Thus, Theorem~\ref{thm: structureLog} can be viewed as a motivic lift of Beilinson's theorem for~$\mathbb{G}_m$. As the referee pointed out to us, it was also proved using a different language and in a more abstract setting by Levine \cite[Proposition 9.3.3]{levinetubular} and Ayoub \cite[Theorem 3.6.44]{ayoubPhD2}.
        \end{remark}

        \begin{remark}
        An alternative strategy for proving Theorem~\ref{thm: structureLog}, which is the one adopted by \cite[Theorem 3.6.44]{ayoubPhD2}, would be as follows. Assuming that we can define by hand the structure of a commutative algebra in $\MT(S)$ on $\mathcal{T}$, the (iso)morphism $\mathcal{K}\to \mathcal{T}_1$ induces for free a morphism of commutative algebras in $\MT(S)$ from the free commutative algebra $\Sym(\mathcal{K})$ to $\mathcal{T}$, and we are left with proving that it is an isomorphism.
        \end{remark}
        
\section{The polylogarithm motive}\label{sec:2}

    \subsection{Definition}
    
        Let $n\geq 0$ be an integer. We denote by $X_n=\AA^n_S$ the $n$-dimensional affine space over $S$, with coordinates $(t_1,\ldots,t_n)$ and projection map $p_n\colon X_n\rightarrow S$. We introduce the following closed $S$\nobreakdash-subschemes of $X_n$:
        $$A_n= \{1-zt_1\cdots t_n=0\} \hspace{.3cm}\textnormal{ and }\hspace{.3cm} B_n = \{t_1(1-t_1)\cdots t_n(1-t_n)=0\}.$$
        
        \begin{definition}
        The $\supth{n}$ \emph{polylogarithm motive} is the relative cohomology motive
        $$\mathcal{L}_n = \M(X_n\setminus A_n, B_n\setminus A_n\cap B_n)\, [n] \;\; \in \DM(S).$$
        \end{definition}
        
        We will see in Theorem~\ref{thm: short exact sequence L} below that $\mathcal{L}_n$ actually belongs to the abelian category $\MT(S)$. For now let us check the $n=0$ case: Since $A_0=B_0=\varnothing\subset X_0=S$, we have $\mathcal{L}_0=\QQ_S(0) \in \MT(S)$.
        
        We introduce closed $S$-subschemes $B'_n=\{t_1(1-t_1)\cdots t_{n-1}(1-t_{n-1})t_n=0\}$ and $Y_n=\{t_n=1\}$ of~$X_n$, so that $B_n=B'_n\cup Y_n$. The pair $(Y_n\setminus Y_n\cap A_n, Y_n\cap B'_n\setminus Y_n\cap B'_n\cap A_n)$ is naturally identified with the pair~$(X_{n-1}\setminus A_{n-1},B_{n-1}\setminus A_{n-1}\cap B_{n-1})$, and hence there is a partial boundary morphism (see Definition~\ref{defi: partial boundary morphism}) along~$Y_n$ denoted by
        $$\lambda_n\colon\mathcal{L}_{n-1}\longrightarrow \mathcal{L}_n.$$
        These morphisms give rise to an inductive system $\mathcal{L}$.
        
        \begin{definition}
        The inductive system $\mathcal{L}$ is called the \emph{polylogarithmic system}.
        \end{definition}
    
    \subsection{The structure of the polylogarithm motive}
    
        \begin{theorem}\label{thm: short exact sequence L}
        The $\supth{n}$ polylogarithm motive $\mathcal{L}_n$ is an object of the category $\MT(S)$ and fits into a short exact sequence
        \begin{equation}\label{eq: short exact sequence polylogarithm}
        0 \longrightarrow \QQ_S(0) \longrightarrow \mathcal{L}_n \longrightarrow \Sym^{n-1}(\mathcal{K})(-1) \longrightarrow 0.
        \end{equation}
        \end{theorem}
        
        \begin{proof} 
        We apply Proposition~\ref{prop: appendix residue} by noting that $(X_n,A_n,B_n)$ is a triple which is locally of product type because $A_n\cup B_n$ is a normal crossing divisor, and that $A_n$ is a smooth divisor in $X_n$. We obtain a distinguished triangle in $\DM(S)$:
        \begin{equation}\label{eq: triangle polylogarithm motive proof}
        \M(X_n,B_n)\,[n] \longrightarrow \mathcal{L}_n \longrightarrow \M(A_n,A_n\cap B_n)\,(-1)\,[n-1] \stackrel{+1}{\longrightarrow}.
        \end{equation}
        The leftmost term of \eqref{eq: triangle polylogarithm motive proof} is easy to compute. The pair $(X_n,B_n)$ is the $\supth{n}$ Cartesian power of the pair~$(X_1,B_1)$, so the K\"{u}nneth formula (see Proposition~\ref{prop: appendix kunneth}) gives an isomorphism in $\MT(S)$:
        $$\M(X_n,B_n)\,[n]\simeq \left( \M(X_1,B_1)[1] \right)^{\otimes n}.$$
        We compute $\M(X_1,B_1)=\M(\mathbb{A}^1_S,\{0,1\})$ thanks to the distinguished triangle (see Proposition~\ref{prop: appendix boundary triangle})
        $$\M(\{1\})\longrightarrow \M(\mathbb{A}^1_S,\{0,1\})\,[1] \longrightarrow \M(\mathbb{A}^1_S,\{0\})\,[1]\stackrel{+1}{\longrightarrow}.$$
        The relative cohomology motive $\M(\AA^1_S,\{0\})$ vanishes because homotopy invariance implies that \hbox{$\M(\AA^1_S) \to \M(\{0\})$} is an isomorphism. Therefore, we have an isomorphism $\M(X_1,B_1)[1]\simeq \M(\{1\})\simeq \QQ_S(0)$, which leads to an isomorphism $\M(X_n,B_n)[n]\simeq \QQ_S(0)$. The rightmost term of \eqref{eq: triangle polylogarithm motive proof} is, by definition, $\mathcal{T}_{n-1}(-1),$ and hence it is isomorphic to $\Sym^{n-1}(\mathcal{K})(-1)$ by Theorem~\ref{thm: structureLog}. Since both~$\QQ_S(0)$ and $\Sym^{n-1}(\mathcal{K})(-1)$ are objects of the category $\MT(S)$, the claim follows.
        \end{proof}
        
        \begin{theorem}\label{thm: short exact sequence L inductive}
        The short exact sequence \eqref{eq: short exact sequence polylogarithm} is compatible with the inductive systems in the sense that we have for each $n\geq 1$ a commutative diagram in $\MT(S)$ with exact rows:
        $$\diagram{
        0 \ar[r] & \QQ_S(0) \ar[r] \ar@{=}[d] & \mathcal{L}_{n-1} \ar[r] \ar[d]^{\lambda_n} & \Sym^{n-2}(\mathcal{K})(-1) \ar[r] \ar[d]^{\iota_{n-1}(-1)} & 0 \\
        0 \ar[r] & \QQ_S(0) \ar[r]  & \mathcal{L}_n \ar[r]  & \Sym^{n-1}(\mathcal{K})(-1) \ar[r]  & 0\rlap{.} 
        }$$
        \end{theorem}
        
        \begin{proof}
        The first square commutes by the functoriality of partial boundary morphisms (see Proposition~\ref{prop: appendix functoriality boundary}). The second square commutes by the compatibility between residue morphisms and partial boundary morphisms (see Proposition~\ref{prop:appendixresidueboundary}).
        \end{proof}
        
        From Theorem~\ref{thm: short exact sequence L inductive} we see that the transition morphisms $\lambda_n$ fit in short exact sequences
        \begin{equation}\label{eq: short exact sequence lambda n}
        \diagram{
        0 \ar[r]& \mathcal{L}_{n-1} \ar[r]^-{\lambda_n} & \mathcal{L}_n \ar[r]^-{\Res} & \QQ(-n) \ar[r]& 0, 
        }
        \end{equation}
        where $\Res$ is the composition of the residue morphism $\mathcal{L}_n\to \M(A_n,A_n\cap B_n)(-1)[n-1]$ with 
        $$
        \M(A_n,A_n\cap B_n)(-1)[n-1]=\H^{n-1}(\M(A_n,A_n\cap B_n))(-1) \longrightarrow \H^{n-1}(\M(A_n))(-1)\simeq \QQ(-n).$$
        
    \subsection{de Rham realization}
    
    We explicitly compute the de Rham realization of $\mathcal{L}_n$, which is an algebraic vector bundle with flat connection on $S_\QQ=\PP^1_\QQ\setminus\{0,1,\infty\}$ that we denote by $(\mathcal{V}_n,\nabla)$. We will first see that $\mathcal{V}_n$ is the trivial vector bundle of rank $n+1$ by exhibiting an explicit basis, and compute the connection. For a smooth variety~$Y$ over $S_\QQ$, we denote by $p_Y\colon Y\to S_\QQ$ the structure morphism and consider the complex of sheaves~$\Omega^\bullet(Y/S_\QQ)=\mathrm{R}^0(p_Y)_*\Omega^\bullet_{Y/S_\QQ}$ on $S_\QQ$ whose local sections are global algebraic forms on the fibers of~$Y$ over~$S_\QQ$. Since $X_n\setminus A_n$ is affine over $S_\QQ$, the vector bundle $\mathcal{V}_n$ is the $\supth{n}$ cohomology sheaf of the total complex of the double complex
    \begin{equation}\label{eq: complex computing de Rham realization}
    0\longrightarrow \Omega^\bullet\left(\left(X_n\setminus A_n\right)/S_\QQ\right) \longrightarrow \bigoplus_{i}\, \Omega^\bullet\left(\left(B_n^i\setminus A_n\cap B_n^i\right)/S_\QQ\right) \longrightarrow \cdots, 
    \end{equation}
    which involves global algebraic forms on $X_n\setminus A_n$, on the irreducible components $B_n^i\setminus A_n\cap B_n^i$ of~$B_n\setminus A_n\cap B_n$, and on their multiple intersections.

    The following construction is inspired by \cite[Section~3]{dupontoddzeta}.  Recall that the \textit{Eulerian polynomials} $E_r(x)$ are defined, for $ r \geq 0$, by the relation
    \begin{equation}\label{eq: definition Eulerian}
    \sum_{j\geq 0}(j+1)^rx^j = \frac{E_r(x)}{(1-x)^{r+1}}.
    \end{equation}
    The first examples are given by $E_0(x)=E_1(x)=1$, $E_2(x)=1+x$, $E_3(x)=1+4x+x^2$. They satisfy the recurrence relation
    \begin{equation}\label{eq:recurrence}
    E_{r+1}(x)=x(1-x)E'_r(x)+(1+rx)E_r(x). 
    \end{equation}
    For $n\geq 0$, we define differential forms  
    $$\omega_n^{(0)}=\d t_1\cdots \d t_n$$ 
    and, for each $k=1, \ldots, n$,
    \begin{displaymath}
    \omega_n^{(k)}=\frac{z\, E_{n-k}(zt_1\cdots t_n)}{(1-zt_1\cdots t_n)^{n-k+1}} \,\d t_1\cdots \d t_n.
    \end{displaymath}
    They are global sections of the sheaf $\Omega^n((X_n\setminus A_n)/S_\QQ)$. Since $X_n\setminus A_n$ has relative dimension $n$ over~$S$, they automatically define global sections of the vector bundle $\mathcal{V}_n$.
    
    \begin{proposition}\label{prop: de rham basis}
    The classes of $\omega_n^{(0)},\ldots,\omega_n^{(n)}$ form a basis of the algebraic vector bundle $\mathcal{V}_n$, and the connection $\nabla\colon \mathcal{V}_n\to \mathcal{V}_n\otimes_{\mathcal{O}_S}\Omega^1_{S_\QQ/\QQ}$ satisfies
    $$\nabla\left(\left[\omega_n^{(0)}\right]\right)=0, \quad \nabla\left(\left[\omega_n^{(1)}\right]\right) = \left[\omega_n^{(0)}\right]\otimes \frac{\d z}{1-z}, \quad \nabla\left(\left[\omega_n^{(k)}\right]\right) = [\omega_n^{(k-1)}]\otimes \frac{\d z}{z} \quad (k=2,\ldots,n).$$
    \end{proposition}
    
    \begin{proof}
    We proceed by induction on $n$. The statement is clear for $n=0$. For the induction step, with $n\geq 1$, we use the following short exact sequence of algebraic vector bundles with flat connections on $S$, which is the de Rham realization of \eqref{eq: short exact sequence lambda n}:
    \begin{equation}\label{eq: short exact sequence lambda n de Rham}
    0 \lra (\mathcal{V}_{n-1},\nabla) \xrightarrow{\lambda_{n,\dR}}  (\mathcal{V}_n,\nabla) \xrightarrow{\Res_\dR} \left(\mathcal{H}^{n-1}_\dR\left(A_n/S_\QQ\right),\nabla\right)\simeq (\mathcal{O}_{S_\QQ},\d)\lra 0.
    \end{equation}
    The morphism $\lambda_{n,\dR}$ is induced by the inclusion inside \eqref{eq: complex computing de Rham realization} of the similar complex computing~$\mathcal{V}_{n-1}$. Therefore, for each global section $\omega$ of  $\Omega^{n-1}((X_{n-1}\setminus A_{n-1})/S_\QQ)$, we have
    $$
    \lambda_{n,\dR}([\omega]) = -[\d\eta],
    $$
    where $\eta$ is any global section of $\Omega^n((X_n\setminus A_n)/S_\QQ)$ such that $\eta_{|\{t_n=1\}}=\omega$ and $\eta$ vanishes when restricted to~$\{t_n=0\}$ and all the $\{t_i=0\}$ and $\{t_i=1\}$ for $i=1,\ldots,n-1$. 
    \begin{enumerate}
    \item For $\omega=\omega_{n-1}^{(0)}$, we may choose $\eta=t_n\, \d t_1\cdots \d t_{n-1}$, and we get 
    \begin{equation}\label{eq: what lambda does in weight 0}
    \lambda_{n,\dR}\left(\left[\omega_{n-1}^{(0)}\right]\right) = (-1)^n \left[\omega_n^{(0)}\right].
    \end{equation}
    Since $\lambda_{n,\dR}$ commutes with the connections, the induction hypothesis implies that 
    \begin{equation}\label{eq: nabla in weight 0}
    \nabla\left(\left[\omega_n^{(0)}\right]\right)=0.
    \end{equation}
    \item For $\omega=\omega_{n-1}^{(k)}$ with $k=1,\ldots,n-1$, we may choose
    $$\eta = \frac{z t_n E_{n-k-1}(zt_1\cdots t_n)}{(1-zt_1\cdots t_n)^{n-k}}\d t_1\cdots \d t_{n-1}.$$
    The recurrence relation \eqref{eq:recurrence} readily implies the equality $\d\eta = (-1)^{n-1} \omega_n^{(k)}$, and hence
    \begin{equation}\label{eq: what lambda does in weight k}
    \lambda_{n,\dR}\left(\left[\omega_{n-1}^{(k)}\right]\right) = (-1)^{n} \left[\omega_n^{(k)}\right] \quad (k=1,\ldots,n-1).
    \end{equation}
    Since $\lambda_{n,\dR}$ commutes with the connections, the induction hypothesis implies that 
    \begin{equation}\label{eq: nabla in weight k}
    \nabla\left(\left[\omega_n^{(k)}\right]\right) = \left[\omega_n^{(k-1)}\right]\otimes \frac{\d z}{z} \quad (k=2,\ldots,n-1).
    \end{equation}
    \item The induction hypothesis also implies that
    $$\nabla\left(\left[\omega_n^{(1)}\right]\right) = \left[\omega_n^{(0)}\right]\otimes \frac{\d z}{1-z}$$
    if $n\geq 2$, and we need to treat the case of $[\omega_1^{(1)}]$ by hand. For this, we compute
    $$\frac{\partial}{\partial z}\omega_1^{(1)} - \frac{1}{1-z}\omega_1^{(0)} = \frac{\d t}{(1-zt)^2} - \frac{1}{1-z}\d t = \d\nu \quad \text{with}\ \nu=-\frac{z}{1-z}\frac{t(1-t)}{1-zt}.$$
    Since $\nu$ vanishes at $t=0$ and $t=1$, this means that $[d\nu]=0$ holds in relative cohomology, whence the result:
    \begin{equation}\label{eq: nabla in weight 1}
    \nabla\left(\left[\omega_1^{(1)}\right]\right) = \left[\omega_1^{(0)}\right]\otimes \frac{\d z}{1-z}.
    \end{equation}
    \item The map $\Res_\dR$ appearing in \eqref{eq: short exact sequence lambda n de Rham} is induced by the residue map along the hypersurface \hbox{$A_n=\{zt_1\cdots t_n=1\}$}. Taking the equality
    $$\omega_n^{(n)} = \frac{z\, \d t_1\cdots \d t_n}{1-zt_1\cdots t_n} = (-1)^{n-1} \d\log(1-zt_1\cdots t_n)\wedge \d\log(t_1)\wedge\cdots \wedge \d\log(t_{n-1})$$
    into account, we see that 
    $$\Res\left(\left[\omega_n^{(n)}\right]\right) = (-1)^{n-1} \left[\d\log(t_1)\wedge \cdots \wedge \d\log(t_{n-1})\right],$$
    which is a basis of $\mathcal{H}^{n-1}_\dR(A_n/S_\QQ)\simeq \mathcal{O}_{S_\QQ}$. Along with \eqref{eq: what lambda does in weight 0} and \eqref{eq: what lambda does in weight k}, this implies that the classes of~$\omega_n^{(0)},\ldots,\omega_n^{(n)}$ form a basis of $\mathcal{V}_n$. In view of \eqref{eq: nabla in weight 0}, \eqref{eq: nabla in weight k}, \eqref{eq: nabla in weight 1}, we are left with proving the formula
    $$\nabla\left(\left[\omega_n^{(n)}\right]\right) = \left[\omega_n^{(n-1)}\right]\otimes \frac{\d z}{z}.$$
    It follows from an easy computation:
    $$\frac{\partial}{\partial z}\omega_n^{(n)} = \frac{\partial}{\partial z} \frac{z\,\d t_1\cdots dt_n}{1-zt_1\cdots t_n} = \frac{\d t_1\cdots \d t_n}{(1-zt_1\cdots t_n)^2} = \frac{1}{z}\omega_n^{(n-1)}.$$
    \end{enumerate}
    This concludes the induction step and the proof.
    \end{proof}
    
    \begin{remark}
    One could choose to define $\omega_0^{(n)}$ by the same formula as the other $\omega_k^{(n)}$. This would only result in a base change (gauge transformation) in the connection matrix.
    \end{remark}
    
    \begin{remark}
    One can note that the recurrence relation \eqref{eq:recurrence} implies the identities
    $$\frac{\partial}{\partial z} \omega_n^{(0)} =  0 \quad \text{and} \quad \frac{\partial}{\partial z} \omega_n^{(k)}=\frac{1}{z}\omega_n^{(k-1)} \quad (k=2,\ldots,n)$$
    already at the level of differential forms. However, the relation
    $$\nabla\left(\left[\omega_n^{(1)}\right]\right)=\left[\omega_n^{(0)}\right]\otimes \frac{\d z}{1-z}$$
    is only true at the level of relative cohomology classes.
    \end{remark}
    
    \subsection{Hodge realization}
    
    We first compute the weight and Hodge filtrations on the algebraic vector bundle $\mathcal{V}_n$.
    
    \begin{proposition}\label{prop: de rham weight hodge}
    The weight and Hodge filtrations on $\mathcal{V}_n$ are given, for all integers $k$, by
    $$W_{2k}\mathcal{V}_n = \mathcal{O}_{S_\QQ}\left[\omega_n^{(0)}\right] \oplus\cdots \oplus \mathcal{O}_{S_\QQ}\left[\omega_n^{(k)}\right] \quad \text{and} \quad F^k\mathcal{V}_n = \mathcal{O}_{S_\QQ}\left[\omega_n^{(k)}\right]\oplus\cdots \oplus \mathcal{O}_{S_\QQ}\left[\omega_n^{(n)}\right].$$
    \end{proposition}
    
    \begin{proof}
    For the weight filtration, we first note that the morphism $\lambda_{n,\dR}$ in the short exact sequence~\eqref{eq: short exact sequence lambda n de Rham} is strictly compatible with the weight filtrations on $\mathcal{V}_{n-1}$ and $\mathcal{V}_n$, and that its cokernel $(\mathcal{O}_{S_\QQ},\d)$ is concentrated in weight $2n$. The statement therefore follows by induction on $n$ using the identities \eqref{eq: what lambda does in weight 0} and \eqref{eq: what lambda does in weight k}. Regarding the Hodge filtration, the same inductive reasoning gives the statement provided that we can prove that $[\omega_n^{(n)}]$ belongs to $F^n\mathcal{V}_n$. This is a consequence of the fact that $\omega_n^{(n)}$ is a logarithmic form on some compactification of $X_n\setminus A_n$ over~$S_\QQ$, as in the proof of \cite[Proposition 3.12]{dupontoddzeta}.
    \end{proof}

    \begin{theorem}\label{thm: hodge realization polylog}
    The Hodge realization of $\mathcal{L}_n$ is the $\supth{n}$ polylogarithmic variation of mixed Hodge structures described in the introduction.
    \end{theorem}
    
    \begin{proof}
    By Propositions~\ref{prop: de rham basis} and~\ref{prop: de rham weight hodge}, the analytic vector bundle with flat connection $(\mathcal{V}^{\an},\nabla^{\an})$ on~\hbox{$\mathbb{P}^1(\CC)\setminus\{0,1,\infty\}$} is the one described in the introduction. Therefore, we only need to prove that the rational structure on the Betti realization $\mathcal{L}_{n,\B}$ is the one induced by the period matrix \eqref{eqn:matrixpolylog}. It is enough to prove it for the fiber of $\mathcal{L}_{n,\B}$ at some point $z\in \PP^1(\CC)\setminus\{0,1,\infty\}$. We contemplate the following short exact sequence of local systems on $\PP^1(\CC)\setminus \{0,1,\infty\}$, induced by \eqref{eq: short exact sequence polylogarithm},
    $$
    0 \lra \Sym^{n-1}\left(\mathcal{K}_\B^\vee\right) \lra \mathcal{L}_{n,\B}^\vee \lra \H^n_\B(X_n,B_n) \lra 0,
    $$
    where $\H^n_\B(X_n,B_n)$ is a rank $1$ constant local system. We fix some $z\in\CC\setminus [1,+\infty)$, so that the hypercube~$[0,1]^n$ does not intersect the hypersurface $\{zt_1\cdots t_n=1\}$ and hence defines a relative homology class in the fiber of $\mathcal{L}^\vee_{n,\B}$ at $z$, which lifts the canonical basis of $\H_n^\B(X_n,B_n)$. A basis of~ $\mathcal{L}_{n,\B}^\vee$ at $z$ is therefore obtained by adjoining the class of $[0,1]^n$ to a basis of $\Sym^{n-1}(\mathcal{K}_\B^\vee)$ at~$z$. The following lemma shows that in such a basis, the period matrix of $\mathcal{L}_n$ at $z$ is \eqref{eqn:matrixpolylog}, which concludes the proof.
    \end{proof}
    
    \begin{lemma}
    Let $z\in \CC\setminus[1,+\infty)$. We have the identities
    $$\int_{[0,1]^n}\omega_n^{(0)} = 1 \quad \text{and} \quad \int_{[0,1]^n}\omega_n^{(k)} = \Li_k(z) \quad (k=1,\ldots,n).$$
    \end{lemma}
    
    \begin{proof}
    The first identity is clear. For the second, we use \eqref{eq: definition Eulerian} to compute
    $$
    \int_{[0,1]^n}\omega_n^{(k)}=\sum_{j\geq 0}(j+1)^{n-k}z^{j+1}\int_{[0,1]^n}(t_1\cdots t_n)^j\d t_1\cdots \d t_n = \sum_{j\geq 0}\frac{z^{j+1}}{(j+1)^k} = \Li_k(z).\eqno\qedhere
    $$
    \end{proof}

\appendix

\section{Relative cohomology motives}\label{sec: appendix relative cohomology motives}

    In this appendix, we collect some useful facts on relative cohomology motives. All of these results are well known to experts, but we were not able to find a reference where they are presented in a systematic way. We fix a base scheme $S$ which is assumed to be separated and of finite type over a Noetherian base scheme, and we call \emph{variety} a scheme $X$ over $S$ which is separated and of finite type, for which we denote the structure morphism by~$p_X\colon X\rightarrow S$. We use the traditional notation $i_Z^X\colon Z\hookrightarrow X$ and~\hbox{$j_U^X\colon U\hookrightarrow X$} for open and closed immersions, respectively.

     \subsection{Relative cohomology motives}
     
        We consider pairs $(X, Z)$ consisting of a variety $X$ and a closed subvariety $Z$ of $X$. They form a category in which a morphism from $(X_1,Z_1)$ to~$(X_2,Z_2)$ is a morphism of varieties $f\colon X_1\rightarrow X_2$ such that $f(Z_1)\subset Z_2$.
        
        \begin{definition}\label{defi: appendix relative cohomology motive}
          Let $(X,Z)$ be a pair of varieties. The object
          $$\M(X,Z) =\left(p_X\right)_*\left(j_{X\setminus Z}^X\right)_!\QQ_{X\setminus Z}(0) \;\;\in\; \DM(S)$$
          is called the \emph{relative cohomology motive} of $(X,Z)$. When $Z$ is empty, we set 
$$\M(X)=\M(X,\varnothing)=\left(p_X\right)_*\QQ_X(0) \;\;\in\; \DM(S)$$ 
        and simply call it the \emph{cohomology motive} of $X$.
        \end{definition}
        
        We warn the reader that in the literature such notation is often used for relative \emph{homology} motives instead. By a slight abuse, shifts of $\M(X,Z)$ will also be called relative cohomology motives.
        
        \begin{remark}
        For $S=\Spec(\CC)$, the Betti realization of the relative cohomology motive $\M(X,Z)$ is a complex that computes the relative singular cohomology groups $\H^\bullet(X,Z)=\H^\bullet(X(\CC), Z(\CC); \QQ)$. 
        \end{remark}

        \begin{proposition}\label{prop: appendix relative cohomology motives functorial}
        Relative cohomology motives yield a contravariant functor $(X,Z)\mapsto \M(X,Z)$ from the category of pairs of varieties to $\DM(S)$.
        \end{proposition}
        
        In particular, we have morphisms $\M(X,Z_2)\to \M(X,Z_1)$ for closed subvarieties $Z_1\subset Z_2\subset X$. We will first need to prove a general lemma.
        
        \begin{lemma}\label{lem: appendix relative cohomology motives functorial}
        Let $(X_i,Z_i)$ be pairs as above for $i=1,2,3$.
        \begin{enumerate}
        \item\label{l:arcmf-1} Let $f\colon X_1\rightarrow X_2$ be a morphism such that $f(Z_1)\subset Z_2$. Then we have a morphism of endofunctors of\, $\DM(X_2)$: 
        $$\Phi_f\colon \left(j_{X_2\setminus Z_2}^{X_2}\right)_!\left(j_{X_2\setminus Z_2}^{X_2}\right)^! \longrightarrow f_*\left(j_{X_1\setminus Z_1}^{X_1}\right)_!\left(j_{X_1\setminus Z_1}^{X_1}\right)^!f^*.$$
        \item\label{l:arcmf-2} Let $g\colon X_2\rightarrow X_3$ be another morphism such that $g(Z_2)\subset Z_3$. Then we have an equality of morphisms of endofunctors of\, $\DM(X_3)$: 
        $$\Phi_{g\circ f} = \left(g_*\Phi_fg^*\right)\circ \Phi_g.$$
        \end{enumerate}
        \end{lemma}
        
        \begin{proof}
        \eqref{l:arcmf-1}~ By the adjunction $(f^*,f_*)$, defining $\Phi_f$ is equivalent to defining a morphism
        $$f^*\left(j_{X_2\setminus Z_2}^{X_2}\right)_!\left(j_{X_2\setminus Z_2}^{X_2}\right)^! \longrightarrow \left(j_{X_1\setminus Z_1}^{X_1}\right)_!\left(j_{X_1\setminus Z_1}^{X_1}\right)^!f^*.$$
        Let us set $U_1=f^{-1}(X_2\setminus Z_2)$. By using base change and $j^!=j^*$, we have an isomorphism
        $$f^*\left(j_{X_2\setminus Z_2}^{X_2}\right)_!\left(j_{X_2\setminus Z_2}^{X_2}\right)^! \simeq \left(j_{U_1}^{X_1}\right)_!\left(j_{U_1}^{X_1}\right)^!f^*.$$
        Since $f(Z_1)\subset Z_2$, we have $U_1\subset X_1\setminus Z_1$, and we derive a morphism of functors
$$\left(j_{U_1}^{X_1}\right)_!\left(j_{U_1}^{X_1}\right)^!\simeq \left(j_{X_1\setminus Z_1}^{X_1}\right)_!\left(j_{U_1}^{X_1\setminus Z_1}\right)_!\left(j_{U_1}^{X_1\setminus Z_1}\right)^!\left(j_{X_1\setminus Z_1}^{X_1}\right)^! \rightarrow \left(j_{X_1\setminus Z_1}^{X_1}\right)_!\left(j_{X_1\setminus Z_1}^{X_1}\right)^!,$$
        and combining all this gives the desired morphism.\
        
        \eqref{l:arcmf-2}~ This is a tedious but instructive exercise in the six-functor formalism that we encourage the reader to solve by themselves.
        \end{proof}
        
        \begin{proof}[Proof of Proposition~\ref{prop: appendix relative cohomology motives functorial}]
        We have by definition $\M(X,Z)=\left(p_X\right)_*\left(j_{X\setminus Z}^X\right)_!\left(j_{X\setminus Z}^X\right)^!\left(p_X\right)^*\QQ_S(0)$. For a morphism of pairs $f\colon (X_1,Z_1)\rightarrow (X_2,Z_2)$, we set $\M(f)=\left(p_{X_2}\right)_*\Phi_f\left(p_{X_2}\right)^*\QQ_S(0)$, which is a morphism from $\M(X_2,Z_2)$ to $\M(X_1,Z_1)$, where $\Phi_f$ was defined in part~\eqref{l:arcmf-1} of Lemma~\ref{lem: appendix relative cohomology motives functorial}. Part~\eqref{l:arcmf-2}  of that lemma then implies the equality $\M(g\circ f) = \M(f)\circ \M(g)$. 
        \end{proof}

    \subsection{Partial boundary morphisms}
    
        \begin{proposition}\label{prop: appendix boundary triangle}
        Let $X$ be a variety, let $Y$ and $Z'$ be closed subvarieties of\, $X$, and set $Z=Z'\cup Y$. There is a distinguished triangle in $\DM(S)$: 
        \begin{equation}\label{eq: appendix boundary triangle}
        \M(Y,Y\cap Z')[-1] \longrightarrow \M(X,Z) \longrightarrow \M(X,Z') \stackrel{+1}{\longrightarrow}.
        \end{equation}
        \end{proposition}
        
        We first need to prove a general lemma.
        
        \begin{lemma}\label{lem: appendix boundary triangle general}
        Let $X$ be a variety, and let $Y$, $Z$, and $Z'$ be closed subvarieties of\, $X$ with $Z=Z'\cup Y$. There is a distinguished triangle of endofunctors of\, $\DM(X)$: 
        \begin{equation*}
        \left(i_Y^X\right)_*\left(j_{Y\setminus Y\cap Z'}^{Y}\right)_!\left(j_{Y\setminus Y\cap Z'}^{Y}\right)^!\left(i_Y^X\right)^*[-1] \longrightarrow \left(j_{X\setminus Z}^X\right)_!\left(j_{X\setminus Z}^X\right)^! \longrightarrow \left(j_{X\setminus Z'}^X\right)_!\left(j_{X\setminus Z'}^X\right)^!\stackrel{+1}{\longrightarrow}.
        \end{equation*}
        \end{lemma}
        
        \begin{proof}
        We consider the distinguished (localization) triangle 
        $i_*i^*[-1]\rightarrow j_!j^! \rightarrow 1 \stackrel{+1}{\rightarrow}$
        for $i=i_{Y\setminus Y\cap Z'}^{X\setminus Z'}$ and $j=j_{X\setminus Z}^{X\setminus Z'}$. Composing it on the left by $\left(j_{X\setminus Z'}^X\right)_!$ and on the right by $\left(j_{X\setminus Z'}^X\right)^!$ gives rise to a distinguished triangle of endofunctors of $\DM(X)$:
        $$\left(j_{X\setminus Z'}^{X\setminus Z}\right)_!\left(i_{Y\setminus Y\cap Z'}^{X\setminus Z'}\right)_*\left(i_{Y\setminus Y\cap Z'}^{X\setminus Z'}\right)^*\left(j_{X\setminus Z'}^{X\setminus Z}\right)^![-1] \longrightarrow \left(j_{X\setminus Z}^X\right)_!\left(j_{X\setminus Z}^X\right)^! \longrightarrow \left(j_{X\setminus Z'}^X\right)_!\left(j_{X\setminus Z'}^X\right)^!\stackrel{+1}{\longrightarrow}.$$
        The result follows from the isomorphisms
        $$\left(j_{X\setminus Z'}^{X\setminus Z}\right)_!\left(i_{Y\setminus Y\cap Z'}^{X\setminus Z'}\right)_* \simeq \left(i_Y^X\right)_*\left(j_{Y\setminus Y\cap Z'}^{Y}\right)_! \quad\text{and}\quad \left(i_{Y\setminus Y\cap Z'}^{X\setminus Z'}\right)^*\left(j_{X\setminus Z'}^{X\setminus Z}\right)^! \simeq \left(j_{Y\setminus Y\cap Z'}^{Y}\right)^!\left(i_Y^X\right)^*.\eqno\qedhere$$
        \end{proof}
        
        \begin{proof}[Proof of Proposition~\ref{prop: appendix boundary triangle}]
          This follows from the distinguished triangle of Lemma~\ref{lem: appendix boundary triangle general} by evaluating the endofunctors at $\QQ_X(0)=\left(p_X\right)^*\QQ_S(0)$ and applying $\left(p_X\right)_*$. 
        \end{proof}
        
        \begin{remark}
        If $Z'$ is empty, then the triangle \eqref{eq: appendix boundary triangle} simply reads 
        $$\M(Z)[-1] \longrightarrow \M(X,Z)\longrightarrow \M(X) \stackrel{+1}{\longrightarrow}.$$
        For $S=\Spec(\mathbb{C})$, its Betti realization gives rise to the long exact sequence in relative singular cohomology for the pair $(X,Z)$:
        \begin{equation}\label{eqn:exactsequencepairs}
        \cdots \longrightarrow \H^{\bullet-1}(Z) \longrightarrow \H^\bullet(X,Z) \longrightarrow \H^\bullet(X) \longrightarrow \cdots
        \end{equation}
        The morphism $\H_\bullet(X,Z)\rightarrow \H_{\bullet-1}(Z)$, which is dual to the morphism appearing in the long exact sequence~\eqref{eqn:exactsequencepairs}, computes the boundary of a relative cycle. In general, the triangle \eqref{eq: appendix boundary triangle} gives rise to a long exact sequence
        $$\cdots \longrightarrow \H^{\bullet-1}(Y,Y\cap Z') \longrightarrow \H^\bullet(X,Z) \longrightarrow \H^\bullet(X,Z') \longrightarrow \cdots\ ,$$
        which can be derived from the long exact sequence in relative cohomology for $Z'\subset Z\subset X$ along with the excision isomorphism, see \cite[Proposition 2.22]{hatcherbook}, 
        $$\H^{\bullet-1}(Z,Z')=\H^{\bullet-1}(Y\cup Z', Z')\simeq \widetilde{\H}^{\bullet-1}(Y\cup Z'/Z') \simeq \widetilde{\H}^{\bullet-1}(Y/Y\cap Z')\simeq \H^{\bullet-1}(Y,Y\cap Z').$$
        The morphism $\H_\bullet(X,Z)\rightarrow \H_{\bullet-1}(Y,Y\cap Z')$, dual to the morphism appearing in that long exact sequence, computes \enquote{the $Y$-component of the boundary of a relative cycle.} This justifies the following terminology.
        \end{remark}
        
        \begin{definition}\label{defi: partial boundary morphism}
        The morphism 
        $$\M(Y,Y\cap Z')[-1] \longrightarrow \M(X,Z)$$ 
        appearing in the triangle \eqref{eq: appendix boundary triangle} is called a \emph{partial boundary morphism} along $Y$.
        \end{definition}
        
        We now prove that partial boundary morphisms are functorial.

        \begin{proposition}\label{prop: appendix functoriality boundary}
        For $i=1,2$, let $X_i, Y_i, Z_i, Z_i'$ be as in Proposition~\ref{prop: appendix boundary triangle}, and let $f\colon X_1\rightarrow X_2$ be a morphism such that $f(Y_1)\subset Y_2$ and  $f(Z'_1)\subset Z'_2$. Then we have the following commutative diagram in~$\DM(S)$, where the horizontal arrows are partial boundary morphisms and the vertical arrows are induced by the functoriality of relative cohomology motives:
        $$\diagram{
        \M(Y_1,Y_1\cap Z'_1)[-1] \ar[r] & \M(X_1,Z_1) \\
        \M(Y_2,Y_2\cap Z'_2)[-1] \ar[u]\ar[r] & \M(X_2,Z_2). \ar[u]
        }$$
        \end{proposition}
        
        We first need to prove a general lemma.
    
        \begin{lemma}\label{lem: appendix general lemma boundary functorial}
        In the setting of Proposition~\ref{prop: appendix functoriality boundary}, let us write $j_i=j_{X_i\setminus Z_i}^{X_i}$ and $j'_i=j_{X_i\setminus Z'_i}^{X_i}$ for $i=1,2$. There is a morphism
        $$\diagram{
f_*\left(i_{Y_1}^{X_1}\right)_*\left(j_{Y_1\setminus Y_1\cap Z'_1}^{Y_1}\right)_! \left(j_{Y_1\setminus Y_1\cap Z'_1}^{Y_1}\right)^! \left(i_{Y_1}^{X_1}\right)^*f^* [-1] \ar[r]& f_*\left(j_1\right)_!\left(j_1\right)^!f^* \ar[r]& f_*\left(j'_1\right)_!\left(j'_1\right)^!f^* \ar[r]^-{+1}& \\
\left(i_{Y_2}^{X_2}\right)_*\left(j_{Y_2\setminus Y_2\cap Z'_2}^{Y_2}\right)_!\left(j_{Y_2\setminus Y_2\cap Z'_2}^{Y_2}\right)^!\left(i_{Y_2}^{X_2}\right)^* [-1] \ar[u]\ar[r]& \left(j_2\right)_!\left(j_2\right)^! \ar[u]\ar[r]& \left(j'_2\right)_!\left(j'_2\right)^! \ar[u]\ar[r]^-{+1}&
        }$$ 
        between the distinguished triangles from Lemma~\ref{lem: appendix boundary triangle general}, in which the vertical arrows are induced by the maps $\Phi_f$ from part~\eqref{l:arcmf-1} of Lemma~\ref{lem: appendix relative cohomology motives functorial}.
        \end{lemma}
        
        \begin{proof}
        The above diagram is composed of three squares, the third one having horizontal arrows marked~$+1$. The second and third squares commute as special cases of Lemma~\ref{lem: appendix relative cohomology motives functorial}\eqref{l:arcmf-2}. By \cite[Proposition 1.1.9]{bbdg}, this implies that the first square commutes since we have
        $$\Hom\left(\left(j_{X_2\setminus Z_2}^{X_2}\right)_!\left(j_{X_2\setminus Z_2}^{X_2}\right)^!, f_*\left(i_{Y_1}^{X_1}\right)_*\left(j_{Y_1\setminus Y_1\cap Z'_1}^{Y_1}\right)_! \left(j_{Y_1\setminus Y_1\cap Z'_1}^{Y_1}\right)^! \left(i_{Y_1}^{X_1}\right)^*f^*[-1] \right)=0.$$
        This vanishing comes from the adjunction $\left(\left(j_{X_2\setminus Z_2}^{X_2}\right)_!, \left(j_{X_2\setminus Z_2}^{X_2}\right)^!\right)$ and the vanishing
        $$\left(j_{X_2\setminus Z_2}^{X_2}\right)^!f_*\left(i_{Y_1}^{X_1}\right)_* = \left(j_{X_2\setminus Z_2}^{X_2}\right)^! \left(i_{Y_2}^{X_2}\right)_*\left(f_{Y_1}^{Y_2}\right)_*=0.\eqno\qedhere$$
        \end{proof}
        
        \begin{proof}[Proof of Proposition~\ref{prop: appendix functoriality boundary}]
          This follows from the first commutative square of Lemma~\ref{lem: appendix general lemma boundary functorial} by evaluating at~$\QQ_{X_2}(0)=\left(p_{X_2}\right)^*\QQ_S(0)$ and applying $\left(p_{X_2}\right)_*$. 
        \end{proof}
        
    \subsection{The K\"{u}nneth formula}
    
        The category of pairs of varieties is endowed with the product
        \begin{displaymath}
        (X_1,Z_1) \times (X_2,Z_2)=(X_1\times X_2, (Z_1\times X_2) \cup (X_1\times Z_2)), 
        \end{displaymath}
        where all products are implicitly taken over the base scheme $S$. We will need a K\"{u}nneth formula for relative cohomology motives, which holds in great generality over a field but not over a general base (see Remark~\ref{rem: kunneth fails} below). We therefore state a very particular case that will be sufficient for our purposes.
        
        \begin{proposition}\label{prop: appendix kunneth}
        Let $(X_1,D_1)$ and $(X_2,D_2)$ be two pairs consisting of a smooth variety and a strict normal crossing divisor. There is a functorial isomorphism in $\DM(S)$: 
        \[
        \M(X_1,D_1)\otimes \M(X_2,D_2) \stackrel{\lowsim}{\longrightarrow} \M(X_1\times X_2, (D_1\times X_2) \cup (X_1\times D_2)). 
        \]
        \end{proposition}
                    
        \begin{proof}
        We denote by $\boxtimes\colon\DM(X_1)\times \DM(X_2)\rightarrow \DM(X_1\times X_2)$ the external tensor product defined as 
        \[
        \mathcal{F}_1\boxtimes\mathcal{F}_2=\pi_1^*\mathcal{F}_1\otimes \pi_2^*\mathcal{F}_2,
        \] with $\pi_i\colon X_1\times X_2\rightarrow X_i$ the two projections. We can write 
        $$\QQ_{X_1\times X_2\setminus (D_1\times X_2\cup X_1\times D_2)}(0)=\QQ_{X_1\setminus D_1\times X_2\setminus D_2}(0)  \simeq \QQ_{X_1\setminus D_1}(0)\boxtimes \QQ_{X_2\setminus D_2}(0).$$
        Now the base change isomorphism and the projection formula imply, as in \cite[Lemma 2.2.3]{Kunneth}, that we have 
        $$\left(j_{X_1\setminus D_1 \times X_2\setminus D_2}^{X_1\times X_2}\right)_!\QQ_{X_1\setminus D_1\times X_2\setminus D_2}(0)  \simeq \left(j_{X_1\setminus D_1}^{X_1}\right)_!\QQ_{X_1\setminus D_1}(0) \boxtimes \left(j_{X_2\setminus D_2}^{X_2}\right)_!\QQ_{X_2\setminus D_2}(0).$$
        For objects $\mathcal{F}_i\in \DM(X_i)$, we have a natural morphism 
        \begin{equation}\label{eq: appendix kunneth morphism}
        \left(p_{X_1}\right)_*\mathcal{F}_1\otimes \left(p_{X_2}\right)_*\mathcal{F}_2 \to \left(p_{X_1\times X_2}\right)_*(\mathcal{F}_1\boxtimes \mathcal{F}_2);
        \end{equation}
        see \cite[Section~2.1.19]{Kunneth}. 
        It is an isomorphism if $\mathcal{F}_i=\left(i_{Y_i}^{X_i}\right)_*\QQ_{Y_i}(0)$ for $Y_i\subset X_i$ a closed subvariety that is smooth over~$S$, by~\cite[Lemma 2.1.8 and Proposition 2.1.20]{Kunneth}. By a simple inclusion-exclusion argument using the localization triangles, one sees that $\left(j_{X_i\setminus D_i}^{X_i}\right)_!\QQ_{X_i\setminus D_i}(0)$ is in the triangulated subcategory of $\DM(X_i)$ generated by such $\mathcal{F}_i$, with $Y_i$ an intersection of certain irreducible components of $D_i$. Therefore, \eqref{eq: appendix kunneth morphism} is an isomorphism for $\mathcal{F}_i=\left(j_{X_i\setminus D_i}^{X_i}\right)_!\QQ_{X_i\setminus D_i}(0)$, and the claim follows.
        \end{proof}
        
        \begin{remark}\label{rem: kunneth fails}
        Here is a counterexample to the general K\"{u}nneth formula over a base. Let $S=\AA^1$ be the affine line over some field, and consider $X_1=\AA^1\setminus\{0\}$ and $X_2=\{0\}$ viewed as varieties over~$S$ with structure morphisms $j\colon\AA^1\setminus\{0\} \hookrightarrow \AA^1$ and $i\colon\{0\}\hookrightarrow \AA^1$. The (fiber) product $X_1\times X_2$ is empty, and therefore~$\M(X_1\times X_2)=0$, in contrast with the non-zero 
        $$\M(X_1)\otimes \M(X_2) = j_*\QQ_{\AA^1\setminus\{0\}}(0)\otimes i_*\QQ_{\{0\}}(0) \simeq i_*i^*j_*\QQ_{\AA^1\setminus\{0\}}(0).$$
        \end{remark}
        
        We have the following compatibility between the K\"{u}nneth formula and partial boundary morphisms, whose proof is left as an exercise to the reader. 
        
        \begin{proposition}\label{prop: appendix kunneth boundary}
        Let $(X_1,D_1)$ and $(X_2,D_2)$ be two pairs consisting of a smooth variety and a strict normal crossing divisor. Let us write $D_1=C_1\cup D'_1$, where $C_1$ is an irreducible component of\,~$D_1$ and~$D'_1$ is the union of the remaining irreducible components. Then we have a commutative diagram in~$\DM(S)$ where the horizontal arrows are the K\"{u}nneth isomorphisms from Proposition~\ref{prop: appendix kunneth} and the vertical arrows are partial boundary morphisms along $C_1$ and $C_1\times X_2$, respectively:
        $$\diagram{
        \M(X_1,D_1)\otimes \M(X_2,D_2) \ar[r]^-{\sim} & \M(X_1\times X_2,D_1\times X_2\cup X_1\times D_2) \\
        \M(C_1,C_1\cap D'_1)\,[1]\otimes \M(X_2,D_2) \ar[u]\ar[r]^-{\sim}& \M(C_1\times X_2,(C_1\cap D'_1)\times X_2\cup C_1\times D_2)\,[1]. \ar[u]
        }$$
        \end{proposition}
        
    \subsection{Relative cohomology motives associated to triples}
    
        We now consider \emph{triples} $(X,A,B)$ consisting of a variety $X$ and two closed subvarieties $A,B\subset X$. In this setting, the relative cohomology motive of the pair~$(X\setminus A,B\setminus A\cap B)$ has a description that is more symmetric in $A$ and~$B$.
        
        \begin{proposition}\label{prop: appendix triple motive}
        For a triple $(X,A,B)$, we have an isomorphism in $\DM(S)$: 
        $$\M(X\setminus A,B\setminus A\cap B) \simeq \left(p_X\right)_* \left(j_{X\setminus A}^X\right)_*\left(j_{X\setminus A}^X\right)^*\left(j_{X\setminus B}^X\right)_!\left(j_{X\setminus B}^X\right)^!\QQ_X(0).$$
        \end{proposition}
        
        \begin{proof}
        This follows from the isomorphism $\left(p_{X\setminus A}\right)_*\simeq \left(p_X\right)_* \left(j_{X\setminus A}^X\right)_*$ and base change:
        $$\left(j_{X\setminus A\cup B}^{X\setminus A}\right)_!\QQ_{X\setminus A\cup B}(0) \simeq \left(j_{X\setminus A\cup B}^{X\setminus A}\right)_!\left(j_{X\setminus A\cup B}^{X\setminus B}\right)^*\left(j_{X\setminus B}^X\right)^!\QQ_X(0) \simeq  \left(j_{X\setminus A}^X\right)^*\left(j_{X\setminus B}^X\right)_! \left(j_{X\setminus B}^X\right)^!\QQ_X(0).\eqno\qedhere$$
        \end{proof}
        
        Even more symmetry is gained if we make an extra geometric assumption.
        
        \begin{definition}
        We say that the triple $(X,A,B)$ is \emph{locally of product type} if, \'{e}tale locally on $X$, it is isomorphic to a triple $(X_1\times X_2, Z_1\times X_2,X_1\times Z_2)$ for pairs $(X_1,Z_1)$ and $(X_2,Z_2)$.
        \end{definition}
        
        If $A$ and $B$ are unions of irreducible components of a normal crossing divisor $D\subset X$ such that no component of $D$ is in both $A$ and $B$, then $(X,A,B)$ is locally of product type.
    
        \begin{proposition}\label{prop: appendix triple commute}
        Let $(X,A,B)$ be a triple that is locally of product type, with $X$ smooth. Then we have an isomorphism in $\DM(X)$: 
        $$\left(j_{X\setminus B}^X\right)_!\left(j_{X\setminus A\cup B}^{X\setminus B}\right)_*\QQ_{X\setminus A\cup B}(0) \stackrel{\lowsimeq}{\longrightarrow} \left(j_{X\setminus A}^X\right)_*\left(j_{X\setminus A\cup B}^{X\setminus A}\right)_!\QQ_{X\setminus A\cup B}(0),$$
        and an isomorphism in $\DM(S)$: 
        $$\M(X\setminus A,B\setminus A\cap B)\simeq \left(p_X\right)_* \left(j_{X\setminus B}^X\right)_!\left(j_{X\setminus B}^X\right)^!\left(j_{X\setminus A}^X\right)_*\left(j_{X\setminus A}^X\right)^*\QQ_X(0).$$
        \end{proposition}
        
        \begin{proof}
        The second isomorphism follows from the first and base change as in the proof of Proposition~\ref{prop: appendix triple motive}. The first morphism corresponds by adjunction to the base change isomorphism 
        $$\left(j_{X\setminus A\cup B}^{X\setminus B}\right)_* \longrightarrow \left(j_{X\setminus A\cup B}^{X\setminus B}\right)_*\left(j_{X\setminus A\cup B}^{X\setminus A}\right)^!\left(j_{X\setminus A\cup B}^{X\setminus A}\right)_! \simeq  \left(j_{X\setminus B}^X\right)^!\left(j_{X\setminus A}^X\right)_*\left(j_{X\setminus A\cup B}^{X\setminus A}\right)_!. $$
        If the triple $(X,A,B)$ is of the type $(X_1\times X_2, Z_1\times X_2, X_1\times Z_2)$ for pairs $(X_1,Z_1)$ and $(X_2,Z_2)$, with~$X_1$ and~$X_2$ smooth, then the first morphism is an isomorphism because both sides are isomorphic to the external tensor product $\left(j_{X_1\setminus Z_1}^{X_1}\right)_*\QQ_{X_1\setminus Z_1}(0)\boxtimes\left(j_{X_2\setminus Z_2}^{X_2}\right)_!\QQ_{X_2\setminus Z_2}(0)$, by the same kind of reasoning as in the proof of Proposition~\ref{prop: appendix kunneth}. The proposition follows by \'{e}tale descent. 
        \end{proof}
        
        \begin{remark}
        Proposition~\ref{prop: appendix triple commute} implies, if $X$ is smooth and proper of dimension $n$ and $(X,A,B)$ is a triple that is locally of product type, that we have Poincar\'{e}--Verdier duality
        \begin{equation}\label{eq: verdier duality rel}
        \mathbb{D}_S\,\M(X\setminus A,B\setminus A\cap B) \simeq \M(X\setminus B,A\setminus A\cap B)\,(n)\,[2n]
        \end{equation}
        whenever $\mathbb{D}_S$ is well  defined.
        \end{remark}
        
    \subsection{Residue morphisms}
        
        \begin{proposition}\label{prop: appendix residue}
        Let $(X,A,B)$ be a triple that is locally of product type, with $X$ smooth and $A$ smooth of pure codimension $c$ in $X$. We have a distinguished triangle in $\DM(S)$: 
        \begin{equation}\label{eq: appendix residue}
        \M(X,B)\longrightarrow \M(X\setminus A,B\setminus A\cap B) \longrightarrow \M(A,A\cap B)\,(-c)\,[-2c+1] \stackrel{+1}{\longrightarrow}.
        \end{equation}
        \end{proposition}
        
        \begin{proof}
        By using the distinguished triangle $1\rightarrow j_*j^*\rightarrow i_!i^![1]\stackrel{+1}{\rightarrow}$ for $i=i_{A\setminus A\cap B}^{X\setminus B}$ and $j=j_{X\setminus A\cup B}^{X\setminus B}$, applying $\left(p_X\right)_*\left(j_{X\setminus B}^X\right)_!$, and using Proposition~\ref{prop: appendix triple commute}, we get a distinguished triangle
        \begin{displaymath}
         \M(X,B) \longrightarrow \M(X\setminus A,B\setminus A\cap B) \longrightarrow \left(p_X\right)_*\left(j_{X\setminus B}^X\right)_!\left(i_{A\setminus A\cap B}^{X\setminus B}\right)_!\left(i_{A\setminus A\cap B}^{X\setminus B}\right)^!\QQ_{X\setminus B}(0)[1] \stackrel{+1}{\longrightarrow}.
        \end{displaymath}
        By purity, we have an isomorphism
        $$\left(i_{A\setminus A\cap B}^{X\setminus B}\right)^!\QQ_{X\setminus B}(0) \simeq \QQ_{A\setminus A\cap B}(-c)\,[-2c],$$
        and the rightmost term of the above triangle is isomorphic to 
$$\left(p_X\right)_*\left(i_A^X\right)_!\left(j_{A\setminus A\cap B}^A\right)_!\QQ_{A\setminus A\cap B}\,(-c)\,[-2c+1]\simeq \M(A,A\cap B)\,(-c)\,[-2c+1].$$
        The proposition follows. 
        \end{proof}
        
        \begin{definition}\label{defi: appendix residue}
        For a triple $(X,A,B)$  that is locally of product type, with $X$ smooth and $A$ smooth of pure codimension $c$ in $X$, the morphism 
        $$\M(X\setminus A,B\setminus A\cap B) \longrightarrow \M(A,A\cap B)\,(-c)\,[-2c+1]$$
        appearing in the triangle \eqref{eq: appendix residue} is called a \emph{residue morphism} along $A$.
        \end{definition}
        
        \begin{remark}
        Under the duality \eqref{eq: verdier duality rel}, the residue morphism along $A$ is exchanged with the boundary morphism
        $$\M(A\setminus A\cap B)[-1]\longrightarrow \M(X\setminus B,A\setminus A\cap B).$$
        There are also \emph{partial} residue morphisms, which we do not need.
        \end{remark}
        
        \begin{remark}
        The name ``residue morphism'' comes from the codimension $c=1$ case, in which (assuming, for simplicity, that $B$ is empty) the de Rham realization of the residue morphism \hbox{$\H^\bullet(X\setminus A)\to \H^{\bullet-1}(A)$} can be computed by the standard residue of logarithmic $1$-forms; see \cite{griffithsharris}.
        \end{remark}
        
        \begin{proposition}\label{prop:appendixresidueboundary}
        Let $(X,A,Z)$ be a triple that is locally of product type, with $X$ smooth and $A$ smooth of pure codimension $c$ in $X$. Let us fix a decomposition $Z=Z'\cup Y$ with $Y$ smooth. Then we have a commutative diagram in $\DM(S)$, where the horizontal arrows are residue morphisms and the vertical arrows are partial boundary morphisms:
        $$\diagram{
        \M(Y\setminus A\cap Y , Y\cap Z'\setminus A\cap Y\cap Z')[-1] \ar[d] \ar[r] & \M(A\cap Y, A\cap Y\cap Z')(-c)[-2c] \ar[d]\\
        \M(X\setminus A,Z\setminus A\cap Z) \ar[r] & \M(A,A\cap Z)(-c)[-2c+1].
        }$$
        \end{proposition}
        
        \begin{proof}
        The assumptions imply that $(Y,A\cap Y,Y\cap Z')$ is locally of product type and that $A\cap Y$ is smooth of codimension $c$ in $Y$; therefore, the upper residue morphism is well defined. The commutativity of the diagram is left as an exercise to the reader.
        \end{proof}
        
\section{Motives of configuration spaces with coefficients}\label{sec: appendix postnikov}

    In this appendix, we present a motivic lift of Getzler's results on mixed Hodge modules on configuration spaces, see \cite{getzlerMHM},  which is used in the proof of Proposition~\ref{prop: iso C D alternating}. This is a special case of the main theorem of \cite{dupontjuteau}.

    \subsection{The Arnol'd modules $\boldsymbol{A_N}$}
    
        Let $N$ be a finite set of cardinality $n\geq 1$, and let $E_N$ denote the graded-commutative $\QQ$-algebra generated by degree $1$ elements $e_{i,j}$, for distinct indices ${i,j\in N}$, subject to the relations $e_{i,j}=e_{j,i}$ and 
        $$e_{i,j}e_{i,k}-e_{i,j}e_{j,k}+e_{i,k}e_{j,k}= 0$$
        for pairwise distinct indices $i,j,k\in N$. This algebra was introduced by Arnol'd, who proved in~\cite{arnold} that it is isomorphic to the rational cohomology algebra of the configuration space of distinct points indexed by~$N$ in $\CC$. We will be interested in its top-degree component.
        
        \begin{definition}
        The \emph{Arnol'd module} $A_N$ is the component of degree $n-1$ of $E_N$.
        \end{definition}

        The $\QQ$-vector space $A_N$ is functorial in $N$ in the sense that a bijection $N\simeq N'$ induces an isomorphism $A_N\simeq A_{N'}$; in particular, $A_N$ is a representation of the symmetric group $\mathfrak{S}_N$.
        
        A classical interpretation of the Arnol'd modules, originally due to Cohen (see \cite[Theorem 12.3]{cohenthesis} and \cite[Theorem 6.1\,(2)]{cohenconfigurationlie}), is that there is an isomorphism of $\mathfrak{S}_N$-representations
        \begin{equation}\label{eq: arnold and lie}
        A_N^\vee \simeq \sgn_N\otimes\Lie(N),
        \end{equation}
        where $\sgn_N$ denotes the sign character and $\Lie(N)$ is the space of Lie words on variables indexed by $N$ that are linear in each variable.

    \subsection{The Arnol'd modules as poset homology groups}
        
        We will use the following interpretation of the Arnol'd module $A_N$ in terms of poset topology; see \cite{orliksolomon, barcelo, hanlonwachs}. Recall that a \emph{partition} of $N$ is a set $\pi$ of disjoint subsets of $N$, called the \emph{blocks} of the partition, whose union is $N$. We denote by $|\pi|$ the cardinality of a partition $\pi$, \textit{i.e.}, the number of blocks. The set $\Pi_N$ of all partitions of $N$ is a poset, where $\pi\leq \pi'$ if and only if $\pi'$ is obtained from $\pi$ by merging blocks. The smallest element $\hat{0}$ of $\Pi_N$ is the partition whose blocks are all singletons, and the largest element $\hat{1}$ is the partition with only one block. Note that the symmetric group $\mathfrak{S}_N$ acts on $\Pi_N$. 
        
        By a special case of \cite{orliksolomon}, the reduced homology of the poset $\Pi_N\setminus\{\hat{0},\hat{1}\}$ is concentrated in top degree~$n-3$, and the corresponding group is isomorphic to $A_N$ as a representation of $\mathfrak{S}_N$:
        \begin{equation}\label{eq: appendix Arnold module as poset homology first}
        \widetilde{\H}_{n-3}\left(\Pi_N\setminus\left\{\hat{0},\hat{1}\right\}\right)\simeq A_N.
        \end{equation}
        For an elementary proof, see \cite[Example 3.14]{petersen}. In \cite[Section~1]{dupontjuteau}, a more natural grading convention is used, where the non-trivial homology group is in degree $n-1$. With the notation of \emph{loc.~cit.}, there is a complex $C^\bullet_{\Pi_N}(\hat{1})$ that computes the (shifted) reduced cohomology of $\Pi_N\setminus\{\hat{0},\hat{1}\}$ with coefficients in $\QQ$, and~\eqref{eq: appendix Arnold module as poset homology first} translates to an $\mathfrak{S}_N$-equivariant quasi-isomorphism 
        \begin{equation}\label{eq: appendix Arnold module as poset homology}
        C^\bullet_{\Pi_N}\left(\hat{1}\right) \simeq A_N^\vee[-n+1].
        \end{equation}
        
        We will need a generalization of this fact to lower intervals in the partition poset. For a partition $\pi\in\Pi_N$, let us consider the complex $C^\bullet(\pi)$ from \cite[Section~1]{dupontjuteau} which computes the (shifted) reduced cohomology of the open interval $(\hat{0},\pi)$ in $\Pi_N$. The closed interval $[\hat{0},\pi]$ is isomorphic to the product of the partition posets $\Pi_B$, for $B\in\pi$, and we therefore get a quasi-isomorphism
        $$C^\bullet(\pi) \simeq \bigotimes_{B\in\pi} C^\bullet_{\Pi_B}\left(\hat{1}\right).$$ 
        Using \eqref{eq: appendix Arnold module as poset homology}, we get an $\mathfrak{S}_N$-equivariant quasi-isomorphism
        \begin{equation}\label{eq: appendix Arnold module as poset homology bis}
        C^\bullet(\pi)\simeq \bigotimes_{B\in\pi} A_B^\vee[-(|B|-1)] \simeq A(\pi)^\vee[-n+|\pi|],
        \end{equation}
        where we have set
        $$A(\pi)=\bigotimes_{B\in \pi}A_B.$$
        
    \subsection{The alternating part of Arnol'd modules}
        
        We will need the following crucial fact about the Arnol'd modules: they do not contain the sign representation.

        \needspace{2\baselineskip}
        \begin{theorem}\label{thm: appendix Arnold sign} 
        If $n=|N|\geq 2$, then $(A_N^\vee)^{\sgn}=0$.
        \end{theorem}
        
        \begin{proof}
        It is enough to prove the theorem after extending scalars to $\CC$. The structure of $A_N^\vee$ as a complex representation of $\mathfrak{S}_N$ is given by the following classical result. 
        Let~$C_n$ be a cyclic subgroup of $\mathfrak{S}_N$ generated by an $n$-cycle, and let $\xi_n$ be a primitive character of~$C_n$. There is an $\mathfrak{S}_N$-equivariant isomorphism
        \begin{displaymath}
        A_N^\vee \simeq \sgn_N\otimes \Ind_{C_n}^{\mathfrak{S}_N}(\xi_n).
        \end{displaymath}
        This fact, or more precisely the isomorphism $\Lie(N)\simeq \Ind_{C_n}^{\mathfrak{S}_N}(\xi_n)$ (see \eqref{eq: arnold and lie}), was proved by Brandt \cite[Theorem III]{brandt} in the language of Schur functors and rediscovered many times, see, \textit{e.g.}, \cite[Proposition 1]{klyachko} and \cite[Section~4.4]{joyal}, or for the homology of partition posets, \cite{hanlon, stanleygroups}. It follows that $(A_N^\vee)^{\sgn}$ is isomorphic to the space of fixed points of the $\mathfrak{S}_N$-representation $\Ind_{C_n}^{\mathfrak{S}_N}(\xi_n)$, which by Frobenius reciprocity is isomorphic to the space of fixed points of the $C_n$\nobreakdash-representa\-tion $\xi_n$. It is zero since $\xi_n$ is a non-trivial character.
        \end{proof}

    \subsection{Motives of configuration spaces with coefficients}
    
        Let $X$ be a scheme (separated and of finite type over a Noetherian base scheme), and let $C_N(X)$ denote the configuration space of points of $X$ indexed by $N$, \textit{i.e.}, the complement in $X^N$ of the union of the diagonals $x_a=x_b$, for distinct $a,b\in N$. This is the open stratum of a natural stratification of $X^N$ indexed by partitions of $N$. Namely, to each $\pi\in\Pi_N$ corresponds a stratum whose Zariski closure is  the closed subscheme $X^\pi\subset X^N$ where $x_a=x_b$ if $a$ and $b$ are in the same block of $\pi$.
        
        We let $j\colon C_N(X)\hookrightarrow X^N$ denote the natural open immersion and let $\mathcal{F}\in \DM(X^N)$ be a motive.  Recall that a Postnikov system in a triangulated category is simply a sequence of distinguished triangles where each triangle has a vertex in common with the next one. The following theorem is a motivic lift of Getzler's results on mixed Hodge modules on configuration spaces; see \cite{getzlerMHM}.
    
        \begin{theorem}\label{thm: postnikov system}
        There is a Postnikov system in the triangulated category $\DM(X^N)$: 
        $$\diagram{
	    0=F^n\ar[rr]&&\cdots\ar[ld]& \cdots\ar[rr] && F^1\ar[rr]\ar[ld] && F^0=j_!j^!\mathcal{F},\ar[ld] \\
	    &G^{n-1}\ar[ul]^{+1}& && G^1\ar[ul]^{+1} && G^0\ar[ul]^{+1} &
	    }$$
	    whose graded objects are given by
	    $$G^k = \bigoplus_{\substack{\pi\in \Pi_N\\ |\pi|=n-k}}  \left(i_{X^\pi}^{X^N}\right)_*\left(i_{X^\pi}^{X^N}\right)^*\mathcal{F}[-k] \;\otimes\; A(\pi)^\vee.$$
	    It is equivariant with respect to the action of\, $\mathfrak{S}_N$.
        \end{theorem}

        \begin{proof}
        This is a special case of the main theorem of \cite{dupontjuteau}, taking the quasi-isomorphism \eqref{eq: appendix Arnold module as poset homology bis} into account. 
        \end{proof}

\section{Computations of extension groups, and comparison with previous work}\label{sec: appendix ext}

The polylogarithm motive is an extension of $\Sym(\mathcal{K})(-1)$ by $\QQ_S(0)$, but in the literature it is sometimes described as an extension of $\Sym(\mathcal{K})$ by $\QQ_S(0)$. In this appendix, we describe how those two types of objects are related. For this, first note that the short exact sequence
$$0\To \QQ_S(0)\To \mathcal{K}\To \QQ_S(-1)\To 0$$
induces a short exact sequence of ind-objects in $\MT(S)$:
\begin{equation}\label{eq: appendix short exact sequence Sym K}
0\To \QQ_S(0) \To \Sym(\mathcal{K}) \To \Sym(\mathcal{K})(-1)\To 0.
\end{equation}
Pullback by the morphism $\Sym(\mathcal{K}) \To \Sym(\mathcal{K})(-1)$ yields a linear map 
$$\Ext^1_{\Ind(\MT(S))}(\Sym(\mathcal{K})(-1),\QQ_S(0)) \To \Ext^1_{\Ind(\MT(S))}(\Sym(\mathcal{K}),\QQ_S(0)),$$
which is computed by the following proposition, proved in Section~\ref{sec: appendix proof ext group} below.

\begin{proposition}\label{prop: appendix ext groups}
The following diagram commutes, where the rows are short exact sequences and the vertical arrows are canonical isomorphisms: 
$$
\xymatrixcolsep{.6cm}\diagram{
0 \ar[r]& \QQ \ar[r]\ar@{=}[d]& \Ext^1_{\Ind(\MT(S))}(\Sym(\mathcal{K})(-1),\QQ_S(0)) \ar[r]\ar[d]^-{\sim}& \Ext^1_{\Ind(\MT(S))}(\Sym(\mathcal{K}),\QQ_S(0)) \ar[r]\ar[d]^-{\sim}& 0 \\
0 \ar[r]& \QQ \ar[r]_{k\mapsto (k,0)}& \QQ\oplus\QQ \ar[r]_{(k_0,k_1)\mapsto k_1}& \QQ \ar[r]& 0\rlap{.}
}$$
Under the canonical isomorphism 
$$\Ext^1_{\Ind(\MT(S))}(\Sym(\mathcal{K})(-1),\QQ_S(0)) \simeq \QQ \oplus \QQ,$$
the class of \eqref{eq: appendix short exact sequence Sym K} corresponds to $(1,0)$, and the class of $\mathcal{L}$ corresponds to $(0,1)$.
\end{proposition}

    In order to help the reader navigate between references, we now make a series of comments about the relation between our setting and \cite{wildeshaus, HW, AyoubOberwolfach, huberkings}.
    \begin{enumerate}
    \item Those references do not consider extensions of $\Sym(\mathcal{K})(-1)$ by $\QQ_S(0)$, but rather extensions of $\Sym(\mathcal{K})$, without the Tate twist, by $\QQ_S(0)$. One can go between one type of extension and the other by using the fact that the short exact sequence in the top row of the diagram of Proposition~\ref{prop: appendix ext groups} canonically splits (because the short exact sequence in the bottom row does and the vertical arrows are canonical).
    \item In fact, those references work in the \emph{dual} setting (consistently with Remark~\ref{rem: dual}) and consider extensions of $\QQ_S(0)$ by the pro-object $\Sym(\mathcal{K}^\vee)$.
    \item As explained in Section~\ref{sec:1}, the Kummer motive $\mathcal{K}$ is the restriction to $S$ of an object $\mathcal{K}'$ of $\MT(\GG_{m,\ZZ})$. What is denoted by $\mathcal{K}$ in most references is the dual of our $\mathcal{K}'$, and the pro-object consisting of its symmetric powers is denoted by $\mathcal{L}og$.
    \item In the work of Huber--Kings \cite[Section~6.4]{huberkings}, the extension corresponding to the polylogarithm motive is denoted by $\overline{\mathrm{pol}}$. The more ``primitive'' extension denoted by $\mathrm{pol}$ there does not seem to have an incarnation in the setting of this paper.
    \end{enumerate}

\subsection{Extensions via residues}
We start by explaining how ``residue'' morphisms control certain extension groups. We introduce the following commutative diagrams: 
$$\diagram{
 S \ar@{^(->}[r]^-{j} \ar[dr]_{a} & \GG_{m,\ZZ} \ar[d]_-{b} & \{1\}, \ar@{_(->}[l]_-{i} \ar@{=}[dl]\\
& \Spec(\ZZ) &
}
\qquad \qquad 
\diagram{
 S \ar@{^(->}[r]^-{j'} \ar[dr]_{a} & \AA^1_{\ZZ} \ar[d]_-{b'} & \{0,1\}. \ar@{_(->}[l]_-{(i_0,i_1)} \ar[dl]\\
& \Spec(\ZZ) &
}$$
The first part of the next proposition is Ayoub's argument; see \cite{AyoubOberwolfach}.

\begin{proposition}\label{prop: appendix general ext}
Let $\mathcal{M}$ be an object of\, $\Ind(\MT(S))$.
\begin{enumerate}
\item\label{p:age-1}
  \begin{enumerate}\item\label{p:age-1a}
    Assume that there exists an object $\mathcal{M}'$ of $\Ind(\MT(\GG_{m,\ZZ}))$ such that $\mathcal{M}=j^*\mathcal{M}'$. Then there is a ``residue'' morphism
$$\Res_1\colon \Ext^1_{\Ind(\MT(S))}(\mathcal{M},\QQ_S(0)) \To \Hom_{\Ind(\MT(\ZZ))}(i^*\mathcal{M}',\QQ(-1)),$$
which is functorial in $\mathcal{M}'$. 
    \item\label{p:age-1b} For $\mathcal{M}=\Sym(\mathcal{K})$, this residue induces an isomorphism
\begin{equation}\label{eq: appendix res one}
\Res_1\colon \Ext^1_{\Ind(\MT(S))}(\Sym(\mathcal{K}),\QQ_S(0)) \stackrel{\lowsim}{\To} \QQ.
\end{equation}
\end{enumerate}
\item\label{p:age-2}
  \begin{enumerate}\item\label{p:age-2a} Assume that there exists an object $\mathcal{M}'$ of $\Ind(\MT(\AA^1_\ZZ))$ such that $\mathcal{M}=(j')^*\mathcal{M}'$. Then there is a ``residue'' morphism
\begin{align*}
    (\Res_0,\Res_1)&\colon \Ext^1_{\Ind(\MT(S))}(\mathcal{M},\QQ_S(0)) \\
    &\To \Hom_{\Ind(\MT(\ZZ))}(i_0^*\mathcal{M}',\QQ(-1)) \oplus \Hom_{\Ind(\MT(\ZZ))}(i_1^*\mathcal{M}',\QQ(-1)),
\end{align*}
which is functorial in $\mathcal{M}'$. The map $\Res_1$ agrees with that of the previous point.
    \item\label{p:age-2b} For $\mathcal{M}=\QQ_S(-1)$, this residue induces an isomorphism
\begin{equation}\label{eq: appendix res zero one}
(\Res_0,\Res_1)\colon \Ext^1_{\MT(S)}(\QQ_S(-1),\QQ_S(0)) \stackrel{\lowsim}{\To} \QQ\oplus \QQ.
\end{equation}
    \end{enumerate}\end{enumerate}
\end{proposition}

\begin{proof}
\eqref{p:age-1}\eqref{p:age-1a}~ Since $S$ is smooth of relative dimension $1$ over $\Spec(\ZZ)$, from the purity isomorphism \hbox{$a^!\QQ(0)\simeq \QQ_S(1)[2]$}, we get $\QQ_S(0)\simeq a^!\QQ(-1)[-2]$. By using the adjunction between $a_!$ and $a^!$, we therefore get an isomorphism 
$$\Hom_{\Ind(\DM(S))}(\mathcal{M},\QQ_S(0)[1]) \simeq \Hom_{\Ind(\DM(\ZZ))}(a_!\mathcal{M},\QQ(-1)[-1]).$$
Now writing $a_!\mathcal{M} \simeq b_!j_!j^!\mathcal{M}'$ (since $j$ is an open immersion), the distinguished localization triangle \hbox{$i_*i^*[-1]\to j_!j^!\to 1 \stackrel{+1}{\to}$} gives rise to a distinguished triangle
$$i^*\mathcal{M}'[-1]\To a_!\mathcal{M}\To b_!\mathcal{M}' \stackrel{+1}{\To}.$$
By applying $\Hom_{\Ind(\MT(\ZZ))}(-,\QQ(-1)[-1])$, we get the desired residue map
$$\Hom_{\Ind(\DM(\ZZ))}(a_!\mathcal{M},\QQ(-1)[-1]) \To \Hom_{\Ind(\DM(\ZZ))}(i^*\mathcal{M}',\QQ(-1)),$$
whose kernel and cokernel are, respectively, governed by the groups
$$\Hom_{\Ind(\DM(\ZZ))}(b_!\mathcal{M}',\QQ(-1)[-1]) \quad \mbox{ and } \quad \Hom_{\Ind(\DM(\ZZ))}(b_!\mathcal{M}',\QQ(-1)).$$

\eqref{p:age-1}\eqref{p:age-1b}~ As explained in Section~\ref{sec:1}, the Kummer motive $\mathcal{K}$ is the restriction to $S$ of an object $\mathcal{K}'$ of $\MT(\GG_{m,\ZZ})$ which satisfies $i^*\mathcal{K}'\simeq \QQ(0)\oplus \QQ(-1)$. We therefore have
$$i^*\Sym(\mathcal{K}') \simeq \bigoplus_{n\geq 0}\QQ(-n),$$
and hence the target of the residue morphism $\Res_1$ is $\QQ$. Furthermore, $\Res_1$ is an isomorphism because 
$$b_!\Sym(\mathcal{K}') \simeq \QQ(0)[-1]$$
and because both $\Hom_{\DM(\ZZ)}(\QQ(0),\QQ(-1))$ and $\Hom_{\DM(\ZZ)}(\QQ(0),\QQ(-1)[-1])$ vanish.

\eqref{p:age-2}\eqref{p:age-2a}~  This follows from the same kind of computation as in the previous point. The functoriality and the compatibility with the map $\Res_1$ from the previous point are obvious. This time the kernel and cokernel of the residue morphism $(\Res_0,\Res_1)$ are, respectively, governed by the groups
$$\Hom_{\Ind(\DM(\ZZ))}(b'_!\mathcal{M}',\QQ(-1)[-1]) \quad \mbox{ and } \quad \Hom_{\Ind(\DM(\ZZ))}(b'_!\mathcal{M}',\QQ(-1)).$$

\eqref{p:age-2}\eqref{p:age-2b}~ We set $\mathcal{M}'=\QQ_{\AA^1_\ZZ}(-1)$. The claim follows from the equality $b'_!\QQ_{\mathbb{A}^1_\ZZ}(-1) \simeq \QQ(-2)[-2]$ and the vanishing of the extension groups $\Hom_{\DM(\ZZ)}(\QQ(-1)[-1],\QQ(0))\simeq \Ext^1_{\MT(\ZZ)}(\QQ(-1),\QQ(0))$ and $\Hom_{\DM(\ZZ)}(\QQ(-1)[-2],\QQ(0))\simeq \Ext^2_{\MT(\ZZ)}(\QQ(-1),\QQ(0))$.
\end{proof}

\begin{remark}
An important point in the proof of \eqref{p:age-1}\eqref{p:age-1b} of Proposition~\ref{prop: appendix general ext} is the isomorphism \hbox{$b_!\Sym(\mathcal{K}') \simeq \QQ(0)[-1]$}. In Betti realization, this is a computation of the compactly supported cohomology of $\Sym(\mathcal{K}')$ on $\CC^*$, which can be understood as follows. Recall that for a local system $\mathcal{V}$ on $\mathbb{C}^*$, if~$T\colon \mathcal{V}_1\to \mathcal{V}_1$ denotes the monodromy automorphism, we have isomorphisms
$$\H^1_{\c}(\CC^*,\mathcal{V}) \simeq \operatorname{ker}(T^{-1}-\id) \quad \text{and} \quad \H^2_{\c}(\CC^*,\mathcal{V})\simeq \coker(T^{-1}-\id).$$
For $\mathcal{V}=\Sym^n(\mathcal{K}')$, we have a decomposition $\mathcal{V}_1=\QQ e_0\oplus \cdots \oplus \QQ e_n$, and in that basis
$$T = \exp \setlength{\arraycolsep}{4pt}\def\arraystretch{1}
			 \left(\begin{matrix}
			0 & 1 & & &     &&\\
			 & 0   & 1 & &  & 0 &  \\
			 &      & 0    & & \ddots &  & \\
			 &      &     & & \ddots &  & \\
			 &0&  & & &   0 & 1   \\
			 && & && & 0  \\
			\end{matrix}\right).$$ 
   Therefore, we have
$$\H^1_{\c}(\CC^*,\Sym^n(\mathcal{K}')) \simeq \QQ e_0 \quad \text{and} \quad \H^2_{\c}(\CC^*,\Sym^n(\mathcal{K}'))\simeq \QQ e_n,$$
and passing to the limit,
$$\H^1_{\c}(\CC^*,\Sym(\mathcal{K}')) \simeq \QQ e_0 \quad \text{and} \quad \H^2_{\c}(\CC^*,\Sym(\mathcal{K}'))=0.$$
This is consistent with the isomorphism $b_!\Sym(\mathcal{K}') \simeq \QQ(0)[-1]$.
\end{remark}

\begin{remark}
The isomorphism \eqref{eq: appendix res zero one} can be proved more easily using the relation to $K$-theory:
$$\Ext^1_{\MT(S)}(\QQ_S(-1),\QQ_S(0)) \simeq K_1(S)_\QQ \simeq \left(\ZZ[z,z^{-1},(1-z)^{-1}]\right)^\times\otimes_\ZZ \QQ \simeq \QQ\oplus \QQ.$$
Note that \eqref{eq: appendix res zero one} sends the class of the Kummer extension $\mathcal{K}$ to $(1,0)$. The other basis element can be obtained by pulling back that class via the automorphism $z\mapsto 1-z$.
\end{remark}

\subsection{Proof of Proposition~\ref{prop: appendix ext groups}}\label{sec: appendix proof ext group}

We build a commutative diagram 
$$
\xymatrixcolsep{.6cm}\diagram{
0 \ar[r]& \QQ \ar[r]^-{i} \ar@{=}[d]& \Ext^1_{\Ind(\MT(S))}(\Sym(\mathcal{K})(-1),\QQ_S(0)) \ar[r]\ar[d]^-{R}& \Ext^1_{\Ind(\MT(S))}(\Sym(\mathcal{K}),\QQ_S(0)) \ar[r]\ar[d]^-{\Res_1}& 0 \\
0 \ar[r]& \QQ \ar[r]_{k\mapsto (k,0)}& \QQ\oplus\QQ \ar[r]_{(k_0,k_1)\mapsto k_1}& \QQ \ar[r]& 0
}$$ as follows. The first row results from applying the functor $\Hom(-, \QQ_S(0))$ to the short exact sequence \eqref{eq: appendix short exact sequence Sym K} and using $\Hom(\QQ_S(0),\QQ_S(0))=\QQ$ and the vanishing of $\Hom(\QQ_S(0),\QQ_S(0)[1])$ and $\Hom(\Sym(\mathcal{K}),\QQ_S(0))$. Note that the morphism $i$ sends $1\in\QQ$ to the class of the extension \eqref{eq: appendix short exact sequence Sym K}.

The inclusion $\QQ_S(0)\hookrightarrow\Sym(\mathcal{K})$ induces a morphism
$$\Ext^1_{\Ind(\MT(S)}(\Sym(\mathcal{K})(-1),\QQ_S(0)) \To \Ext^1_{\MT(S)}(\QQ_S(-1),\QQ_S(0)).$$ Composing with \eqref{eq: appendix res zero one} gives rise to the middle vertical arrow $R$. Note that $R$ sends the class of~$\Sym(\mathcal{K})$ to~$(1,0)$ and the class of $\mathcal{L}$ to~$(0,1)$ because $\mathcal{L}_1$ is the ``Kummer motive around $1$,'' obtained by pulling back~$\mathcal{K}$ by $z\mapsto 1-z$. This implies that the leftmost square commutes.

Finally, the rightmost square commutes because of the compatibility of the two maps $\Res_1$ (see Proposition~\ref{prop: appendix ext groups}\eqref{p:age-2}\eqref{p:age-2a}). The claim follows.


\end{document}